\documentclass[11pt,twoside,a4paper]{article}
\usepackage{color,amscd,amsmath,amssymb,amsthm,latexsym,stmaryrd,authblk,mathabx,shuffle,mathrsfs,hyperref,tikz,tkz-tab} 
\usepackage[latin1]{inputenc}
\usepackage[T1]{fontenc}   
\usepackage[english]{babel}
\providecommand{\keywords}[1]{\textbf{\textit{Keywords.}} #1}
\providecommand{\AMSclass}[1]{\textbf{\textit{AMS classification.}} #1}

\input{xy}
\xyoption{all}

\newcommand{\ncun}{
\begin{tikzpicture}[line cap=round,line join=round,>=triangle 45,x=0.3cm,y=0.3cm]
\clip(-0.2,0.) rectangle (0.2,1.);
\draw [line width=0.8pt] (0.,1.)-- (0.,0.);
\end{tikzpicture}}

\newcommand{\ncdeuxun}{
\begin{tikzpicture}[line cap=round,line join=round,>=triangle 45,x=0.3cm,y=0.3cm]
\clip(-0.2,0.) rectangle (1.2,1.);
\draw [line width=0.8pt] (0.,1.)-- (0.,0.);
\draw [line width=0.8pt] (0.,0.)-- (1.,0.);
\draw [line width=0.8pt] (1.,0.)-- (1.,1.);
\end{tikzpicture}}

\newcommand{\ncdeuxdeux}{
\begin{tikzpicture}[line cap=round,line join=round,>=triangle 45,x=0.3cm,y=0.3cm]
\clip(-0.2,0.) rectangle (1.2,1.);
\draw [line width=0.8pt] (0.,1.)-- (0.,0.);
\draw [line width=0.8pt] (1.,0.)-- (1.,1.);
\end{tikzpicture}}

\newcommand{\nctroisun}{
\begin{tikzpicture}[line cap=round,line join=round,>=triangle 45,x=0.3cm,y=0.3cm]
\clip(-0.2,0.) rectangle (2.2,1.);
\draw [line width=0.8pt] (0.,1.)-- (0.,0.);
\draw [line width=0.8pt] (0.,0.)-- (2.,0.);
\draw [line width=0.8pt] (1.,0.)-- (1.,1.);
\draw [line width=0.8pt] (2.,0.)-- (2.,1.);
\end{tikzpicture}}

\newcommand{\nctroisdeux}{
\begin{tikzpicture}[line cap=round,line join=round,>=triangle 45,x=0.3cm,y=0.3cm]
\clip(-0.2,0.) rectangle (2.2,1.);
\draw [line width=0.8pt] (0.,1.)-- (0.,0.);
\draw [line width=0.8pt] (0.,0.)-- (1.,0.);
\draw [line width=0.8pt] (1.,0.)-- (1.,1.);
\draw [line width=0.8pt] (2.,0.)-- (2.,1.);
\end{tikzpicture}}

\newcommand{\nctroistrois}{
\begin{tikzpicture}[line cap=round,line join=round,>=triangle 45,x=0.3cm,y=0.3cm]
\clip(-0.2,0.) rectangle (2.2,1.);
\draw [line width=0.8pt] (0.,1.)-- (0.,0.);
\draw [line width=0.8pt] (1.,0.)-- (2.,0.);
\draw [line width=0.8pt] (1.,0.)-- (1.,1.);
\draw [line width=0.8pt] (2.,0.)-- (2.,1.);
\end{tikzpicture}}

\newcommand{\nctroisquatre}{
\begin{tikzpicture}[line cap=round,line join=round,>=triangle 45,x=0.3cm,y=0.3cm]
\clip(-0.2,0.) rectangle (2.2,1.5);
\draw [line width=0.8pt] (0.,1.)-- (0.,0.);
\draw [line width=0.8pt] (0.,0.)-- (2.,0.);
\draw [line width=0.8pt] (1.,0.5)-- (1.,1.5);
\draw [line width=0.8pt] (2.,0.)-- (2.,1.);
\end{tikzpicture}}

\newcommand{\nctroiscinq}{
\begin{tikzpicture}[line cap=round,line join=round,>=triangle 45,x=0.3cm,y=0.3cm]
\clip(-0.1,0.) rectangle (2.2,1.);
\draw [line width=0.8pt] (0.,1.)-- (0.,0.);
\draw [line width=0.8pt] (1.,0.)-- (1.,1.);
\draw [line width=0.8pt] (2.,0.)-- (2.,1.);
\end{tikzpicture}}

\newcommand{\ncquatreun}{
\begin{tikzpicture}[line cap=round,line join=round,>=triangle 45,x=0.3cm,y=0.3cm]
\clip(-0.2,0.) rectangle (3.2,1.);
\draw [line width=0.8pt] (0.,1.)-- (0.,0.);
\draw [line width=0.8pt] (0.,0.)-- (3.,0.);
\draw [line width=0.8pt] (1.,0.)-- (1.,1.);
\draw [line width=0.8pt] (2.,0.)-- (2.,1.);
\draw [line width=0.8pt] (3.,0.)-- (3.,1.);
\end{tikzpicture}}

\newcommand{\ncquatredeux}{
\begin{tikzpicture}[line cap=round,line join=round,>=triangle 45,x=0.3cm,y=0.3cm]
\clip(-0.2,0.) rectangle (3.2,1.);
\draw [line width=0.8pt] (0.,1.)-- (0.,0.);
\draw [line width=0.8pt] (0.,0.)-- (2.,0.);
\draw [line width=0.8pt] (1.,0.)-- (1.,1.);
\draw [line width=0.8pt] (2.,0.)-- (2.,1.);
\draw [line width=0.8pt] (3.,0.)-- (3.,1.);
\end{tikzpicture}}

\newcommand{\ncquatretrois}{
\begin{tikzpicture}[line cap=round,line join=round,>=triangle 45,x=0.3cm,y=0.3cm]
\clip(-0.2,0.) rectangle (3.2,1.);
\draw [line width=0.8pt] (0.,1.)-- (0.,0.);
\draw [line width=0.8pt] (1.,0.)-- (3.,0.);
\draw [line width=0.8pt] (1.,0.)-- (1.,1.);
\draw [line width=0.8pt] (2.,0.)-- (2.,1.);
\draw [line width=0.8pt] (3.,0.)-- (3.,1.);
\end{tikzpicture}}

\newcommand{\ncquatrequatre}{
\begin{tikzpicture}[line cap=round,line join=round,>=triangle 45,x=0.3cm,y=0.3cm]
\clip(-0.2,0.) rectangle (3.2,1.5);
\draw [line width=0.8pt] (0.,1.)-- (0.,0.);
\draw [line width=0.8pt] (0.,0.)-- (3.,0.);
\draw [line width=0.8pt] (1.,0.5)-- (1.,1.5);
\draw [line width=0.8pt] (2.,0.)-- (2.,1.);
\draw [line width=0.8pt] (3.,0.)-- (3.,1.);
\end{tikzpicture}}

\newcommand{\ncquatrecinq}{
\begin{tikzpicture}[line cap=round,line join=round,>=triangle 45,x=0.3cm,y=0.3cm]
\clip(-0.2,0.) rectangle (3.2,1.5);
\draw [line width=0.8pt] (0.,1.)-- (0.,0.);
\draw [line width=0.8pt] (0.,0.)-- (3.,0.);
\draw [line width=0.8pt] (1.,0.)-- (1.,1.);
\draw [line width=0.8pt] (2.,0.5)-- (2.,1.5);
\draw [line width=0.8pt] (3.,0.)-- (3.,1.);
\end{tikzpicture}}

\newcommand{\ncquatresix}{
\begin{tikzpicture}[line cap=round,line join=round,>=triangle 45,x=0.3cm,y=0.3cm]
\clip(-0.2,0.) rectangle (3.2,1.);
\draw [line width=0.8pt] (0.,1.)-- (0.,0.);
\draw [line width=0.8pt] (0.,0.)-- (1.,0.);
\draw [line width=0.8pt] (2.,0.)-- (3.,0.);
\draw [line width=0.8pt] (1.,0.)-- (1.,1.);
\draw [line width=0.8pt] (2.,0.)-- (2.,1.);
\draw [line width=0.8pt] (3.,0.)-- (3.,1.);
\end{tikzpicture}}

\newcommand{\ncquatresept}{
\begin{tikzpicture}[line cap=round,line join=round,>=triangle 45,x=0.3cm,y=0.3cm]
\clip(-0.2,0.) rectangle (3.2,1.5);
\draw [line width=0.8pt] (0.,1.)-- (0.,0.);
\draw [line width=0.8pt] (0.,0.)-- (3.,0.);
\draw [line width=0.8pt] (1.,0.5)-- (2.,0.5);
\draw [line width=0.8pt] (1.,0.5)-- (1.,1.5);
\draw [line width=0.8pt] (2.,0.5)-- (2.,1.5);
\draw [line width=0.8pt] (3.,0.)-- (3.,1.);
\end{tikzpicture}}

\newcommand{\ncquatrehuit}{
\begin{tikzpicture}[line cap=round,line join=round,>=triangle 45,x=0.3cm,y=0.3cm]
\clip(-0.2,0.) rectangle (3.2,1.);
\draw [line width=0.8pt] (0.,1.)-- (0.,0.);
\draw [line width=0.8pt] (0.,0.)-- (1.,0.);
\draw [line width=0.8pt] (1.,0.)-- (1.,1.);
\draw [line width=0.8pt] (2.,0.)-- (2.,1.);
\draw [line width=0.8pt] (3.,0.)-- (3.,1.);
\end{tikzpicture}}

\newcommand{\ncquatreneuf}{
\begin{tikzpicture}[line cap=round,line join=round,>=triangle 45,x=0.3cm,y=0.3cm]
\clip(-0.2,0.) rectangle (3.2,1.);
\draw [line width=0.8pt] (0.,1.)-- (0.,0.);
\draw [line width=0.8pt] (1.,0.)-- (2.,0.);
\draw [line width=0.8pt] (1.,0.)-- (1.,1.);
\draw [line width=0.8pt] (2.,0.)-- (2.,1.);
\draw [line width=0.8pt] (3.,0.)-- (3.,1.);
\end{tikzpicture}}

\newcommand{\ncquatredix}{
\begin{tikzpicture}[line cap=round,line join=round,>=triangle 45,x=0.3cm,y=0.3cm]
\clip(-0.2,0.) rectangle (3.2,1.);
\draw [line width=0.8pt] (0.,1.)-- (0.,0.);
\draw [line width=0.8pt] (2.,0.)-- (3.,0.);
\draw [line width=0.8pt] (1.,0.)-- (1.,1.);
\draw [line width=0.8pt] (2.,0.)-- (2.,1.);
\draw [line width=0.8pt] (3.,0.)-- (3.,1.);
\end{tikzpicture}}

\newcommand{\ncquatreonze}{
\begin{tikzpicture}[line cap=round,line join=round,>=triangle 45,x=0.3cm,y=0.3cm]
\clip(-0.2,0.) rectangle (3.2,1.5);
\draw [line width=0.8pt] (0.,1.)-- (0.,0.);
\draw [line width=0.8pt] (0.,0.)-- (2.,0.);
\draw [line width=0.8pt] (1.,0.5)-- (1.,1.5);
\draw [line width=0.8pt] (2.,0.)-- (2.,1.);
\draw [line width=0.8pt] (3.,0.)-- (3.,1.);
\end{tikzpicture}}

\newcommand{\ncquatredouze}{
\begin{tikzpicture}[line cap=round,line join=round,>=triangle 45,x=0.3cm,y=0.3cm]
\clip(-0.2,0.) rectangle (3.2,1.5);
\draw [line width=0.8pt] (0.,1.)-- (0.,0.);
\draw [line width=0.8pt] (1.,0.)-- (3.,0.);
\draw [line width=0.8pt] (1.,0.)-- (1.,1.);
\draw [line width=0.8pt] (2.,0.5)-- (2.,1.5);
\draw [line width=0.8pt] (3.,0.)-- (3.,1.);
\end{tikzpicture}}

\newcommand{\ncquatretreize}{
\begin{tikzpicture}[line cap=round,line join=round,>=triangle 45,x=0.3cm,y=0.3cm]
\clip(-0.2,0.) rectangle (3.2,1.5);
\draw [line width=0.8pt] (0.,1.)-- (0.,0.);
\draw [line width=0.8pt] (0.,0.)-- (3.,0.);
\draw [line width=0.8pt] (1.,0.5)-- (1.,1.5);
\draw [line width=0.8pt] (2.,0.5)-- (2.,1.5);
\draw [line width=0.8pt] (3.,0.)-- (3.,1.);
\end{tikzpicture}}

\newcommand{\ncquatrequatorze}{
\begin{tikzpicture}[line cap=round,line join=round,>=triangle 45,x=0.3cm,y=0.3cm]
\clip(-0.2,0.) rectangle (3.2,1.);
\draw [line width=0.8pt] (0.,1.)-- (0.,0.);
\draw [line width=0.8pt] (1.,0.)-- (1.,1.);
\draw [line width=0.8pt] (2.,0.)-- (2.,1.);
\draw [line width=0.8pt] (3.,0.)-- (3.,1.);
\end{tikzpicture}}

\newcommand{\nciun}[1]{
\begin{tikzpicture}[line cap=round,line join=round,>=triangle 45,x=0.3cm,y=0.3cm]
\clip(-0.2,0.) rectangle (0.2,1.8);
\draw [line width=0.8pt] (0.,1.)-- (0.,0.);
\draw(0.,1.5) node {\tiny #1};
\end{tikzpicture}}

\newcommand{\ncideuxun}[1]{
\begin{tikzpicture}[line cap=round,line join=round,>=triangle 45,x=0.3cm,y=0.3cm]
\clip(-0.2,0.) rectangle (1.2,1.8);
\draw [line width=0.8pt] (0.,1.)-- (0.,0.);
\draw [line width=0.8pt] (0.,0.)-- (1.,0.);
\draw [line width=0.8pt] (1.,0.)-- (1.,1.);
\draw(0.,1.5) node {\tiny #1};
\end{tikzpicture}}

\newcommand{\ncideuxdeux}[2]{
\begin{tikzpicture}[line cap=round,line join=round,>=triangle 45,x=0.3cm,y=0.3cm]
\clip(-0.2,0.) rectangle (1.2,1.8);
\draw [line width=0.8pt] (0.,1.)-- (0.,0.);
\draw [line width=0.8pt] (1.,0.)-- (1.,1.);
\draw(0.,1.5) node {\tiny #1};
\draw(1.,1.5) node {\tiny #2};
\end{tikzpicture}}

\newcommand{\ncitroisun}[1]{
\begin{tikzpicture}[line cap=round,line join=round,>=triangle 45,x=0.3cm,y=0.3cm]
\clip(-0.2,0.) rectangle (2.2,1.8);
\draw [line width=0.8pt] (0.,1.)-- (0.,0.);
\draw [line width=0.8pt] (0.,0.)-- (2.,0.);
\draw [line width=0.8pt] (1.,0.)-- (1.,1.);
\draw [line width=0.8pt] (2.,0.)-- (2.,1.);
\draw(0.,1.5) node {\tiny #1};
\end{tikzpicture}}

\newcommand{\ncitroisdeux}[2]{
\begin{tikzpicture}[line cap=round,line join=round,>=triangle 45,x=0.3cm,y=0.3cm]
\clip(-0.2,0.) rectangle (2.2,1.8);
\draw [line width=0.8pt] (0.,1.)-- (0.,0.);
\draw [line width=0.8pt] (0.,0.)-- (1.,0.);
\draw [line width=0.8pt] (1.,0.)-- (1.,1.);
\draw [line width=0.8pt] (2.,0.)-- (2.,1.);
\draw(0.,1.5) node {\tiny #1};
\draw(2.,1.5) node {\tiny #2};
\end{tikzpicture}}

\newcommand{\ncitroistrois}[2]{
\begin{tikzpicture}[line cap=round,line join=round,>=triangle 45,x=0.3cm,y=0.3cm]
\clip(-0.2,0.) rectangle (2.2,1.8);
\draw [line width=0.8pt] (0.,1.)-- (0.,0.);
\draw [line width=0.8pt] (1.,0.)-- (2.,0.);
\draw [line width=0.8pt] (1.,0.)-- (1.,1.);
\draw [line width=0.8pt] (2.,0.)-- (2.,1.);
\draw(0.,1.5) node {\tiny #1};
\draw(1.,1.5) node {\tiny #2};
\end{tikzpicture}}

\newcommand{\ncitroisquatre}[2]{
\begin{tikzpicture}[line cap=round,line join=round,>=triangle 45,x=0.3cm,y=0.3cm]
\clip(-0.2,0.) rectangle (2.2,2.3);
\draw [line width=0.8pt] (0.,1.)-- (0.,0.);
\draw [line width=0.8pt] (0.,0.)-- (2.,0.);
\draw [line width=0.8pt] (1.,0.5)-- (1.,1.5);
\draw [line width=0.8pt] (2.,0.)-- (2.,1.);
\draw(0.,1.5) node {\tiny #1};
\draw(1.,2.) node {\tiny #2};
\end{tikzpicture}}

\newcommand{\ncitroiscinq}[3]{
\begin{tikzpicture}[line cap=round,line join=round,>=triangle 45,x=0.3cm,y=0.3cm]
\clip(-0.1,0.) rectangle (2.2,1.8);
\draw [line width=0.8pt] (0.,1.)-- (0.,0.);
\draw [line width=0.8pt] (1.,0.)-- (1.,1.);
\draw [line width=0.8pt] (2.,0.)-- (2.,1.);
\draw(0.,1.5) node {\tiny #1};
\draw(1.,1.5) node {\tiny #2};
\draw(2.,1.5) node {\tiny #3};
\end{tikzpicture}}

\newcommand{\nciquatreun}[1]{
\begin{tikzpicture}[line cap=round,line join=round,>=triangle 45,x=0.3cm,y=0.3cm]
\clip(-0.2,0.) rectangle (3.2,1.8);
\draw [line width=0.8pt] (0.,1.)-- (0.,0.);
\draw [line width=0.8pt] (0.,0.)-- (3.,0.);
\draw [line width=0.8pt] (1.,0.)-- (1.,1.);
\draw [line width=0.8pt] (2.,0.)-- (2.,1.);
\draw [line width=0.8pt] (3.,0.)-- (3.,1.);
\draw(0.,1.5) node {\tiny #1};
\end{tikzpicture}}

\newcommand{\nciquatredeux}[2]{
\begin{tikzpicture}[line cap=round,line join=round,>=triangle 45,x=0.3cm,y=0.3cm]
\clip(-0.2,0.) rectangle (3.2,1.8);
\draw [line width=0.8pt] (0.,1.)-- (0.,0.);
\draw [line width=0.8pt] (0.,0.)-- (2.,0.);
\draw [line width=0.8pt] (1.,0.)-- (1.,1.);
\draw [line width=0.8pt] (2.,0.)-- (2.,1.);
\draw [line width=0.8pt] (3.,0.)-- (3.,1.);
\draw(0.,1.5) node {\tiny #1};
\draw(3.,1.5) node {\tiny #2};
\end{tikzpicture}}

\newcommand{\nciquatretrois}[2]{
\begin{tikzpicture}[line cap=round,line join=round,>=triangle 45,x=0.3cm,y=0.3cm]
\clip(-0.2,0.) rectangle (3.2,1.8);
\draw [line width=0.8pt] (0.,1.)-- (0.,0.);
\draw [line width=0.8pt] (1.,0.)-- (3.,0.);
\draw [line width=0.8pt] (1.,0.)-- (1.,1.);
\draw [line width=0.8pt] (2.,0.)-- (2.,1.);
\draw [line width=0.8pt] (3.,0.)-- (3.,1.);
\draw(0.,1.5) node {\tiny #1};
\draw(1.,1.5) node {\tiny #2};
\end{tikzpicture}}

\newcommand{\nciquatrequatre}[2]{
\begin{tikzpicture}[line cap=round,line join=round,>=triangle 45,x=0.3cm,y=0.3cm]
\clip(-0.2,0.) rectangle (3.2,2.3);
\draw [line width=0.8pt] (0.,1.)-- (0.,0.);
\draw [line width=0.8pt] (0.,0.)-- (3.,0.);
\draw [line width=0.8pt] (1.,0.5)-- (1.,1.5);
\draw [line width=0.8pt] (2.,0.)-- (2.,1.);
\draw [line width=0.8pt] (3.,0.)-- (3.,1.);
\draw(0.,1.5) node {\tiny #1};
\draw(1.,2.) node {\tiny #2};
\end{tikzpicture}}

\newcommand{\nciquatrecinq}[2]{
\begin{tikzpicture}[line cap=round,line join=round,>=triangle 45,x=0.3cm,y=0.3cm]
\clip(-0.2,0.) rectangle (3.2,2.3);
\draw [line width=0.8pt] (0.,1.)-- (0.,0.);
\draw [line width=0.8pt] (0.,0.)-- (3.,0.);
\draw [line width=0.8pt] (1.,0.)-- (1.,1.);
\draw [line width=0.8pt] (2.,0.5)-- (2.,1.5);
\draw [line width=0.8pt] (3.,0.)-- (3.,1.);
\draw(0.,1.5) node {\tiny #1};
\draw(2.,2.) node {\tiny #2};
\end{tikzpicture}}

\newcommand{\nciquatresix}[2]{
\begin{tikzpicture}[line cap=round,line join=round,>=triangle 45,x=0.3cm,y=0.3cm]
\clip(-0.2,0.) rectangle (3.2,1.8);
\draw [line width=0.8pt] (0.,1.)-- (0.,0.);
\draw [line width=0.8pt] (0.,0.)-- (1.,0.);
\draw [line width=0.8pt] (2.,0.)-- (3.,0.);
\draw [line width=0.8pt] (1.,0.)-- (1.,1.);
\draw [line width=0.8pt] (2.,0.)-- (2.,1.);
\draw [line width=0.8pt] (3.,0.)-- (3.,1.);
\draw(0.,1.5) node {\tiny #1};
\draw(2.,1.5) node {\tiny #2};
\end{tikzpicture}}

\newcommand{\nciquatresept}[2]{
\begin{tikzpicture}[line cap=round,line join=round,>=triangle 45,x=0.3cm,y=0.3cm]
\clip(-0.2,0.) rectangle (3.2,2.3);
\draw [line width=0.8pt] (0.,1.)-- (0.,0.);
\draw [line width=0.8pt] (0.,0.)-- (3.,0.);
\draw [line width=0.8pt] (1.,0.5)-- (2.,0.5);
\draw [line width=0.8pt] (1.,0.5)-- (1.,1.5);
\draw [line width=0.8pt] (2.,0.5)-- (2.,1.5);
\draw [line width=0.8pt] (3.,0.)-- (3.,1.);
\draw(0.,1.5) node {\tiny #1};
\draw(1.,2.) node {\tiny #2};
\end{tikzpicture}}

\newcommand{\nciquatrehuit}[3]{
\begin{tikzpicture}[line cap=round,line join=round,>=triangle 45,x=0.3cm,y=0.3cm]
\clip(-0.2,0.) rectangle (3.2,1.8);
\draw [line width=0.8pt] (0.,1.)-- (0.,0.);
\draw [line width=0.8pt] (0.,0.)-- (1.,0.);
\draw [line width=0.8pt] (1.,0.)-- (1.,1.);
\draw [line width=0.8pt] (2.,0.)-- (2.,1.);
\draw [line width=0.8pt] (3.,0.)-- (3.,1.);
\draw(0.,1.5) node {\tiny #1};
\draw(2.,1.5) node {\tiny #2};
\draw(3.,1.5) node {\tiny #3};
\end{tikzpicture}}

\newcommand{\nciquatreneuf}[3]{
\begin{tikzpicture}[line cap=round,line join=round,>=triangle 45,x=0.3cm,y=0.3cm]
\clip(-0.2,0.) rectangle (3.2,1.8);
\draw [line width=0.8pt] (0.,1.)-- (0.,0.);
\draw [line width=0.8pt] (1.,0.)-- (2.,0.);
\draw [line width=0.8pt] (1.,0.)-- (1.,1.);
\draw [line width=0.8pt] (2.,0.)-- (2.,1.);
\draw [line width=0.8pt] (3.,0.)-- (3.,1.);
\draw(0.,1.5) node {\tiny #1};
\draw(1.,1.5) node {\tiny #2};
\draw(3.,1.5) node {\tiny #3};
\end{tikzpicture}}

\newcommand{\nciquatredix}[3]{
\begin{tikzpicture}[line cap=round,line join=round,>=triangle 45,x=0.3cm,y=0.3cm]
\clip(-0.2,0.) rectangle (3.2,1.8);
\draw [line width=0.8pt] (0.,1.)-- (0.,0.);
\draw [line width=0.8pt] (2.,0.)-- (3.,0.);
\draw [line width=0.8pt] (1.,0.)-- (1.,1.);
\draw [line width=0.8pt] (2.,0.)-- (2.,1.);
\draw [line width=0.8pt] (3.,0.)-- (3.,1.);
\draw(0.,1.5) node {\tiny #1};
\draw(1.,1.5) node {\tiny #2};
\draw(2.,1.5) node {\tiny #3};
\end{tikzpicture}}

\newcommand{\nciquatreonze}[3]{
\begin{tikzpicture}[line cap=round,line join=round,>=triangle 45,x=0.3cm,y=0.3cm]
\clip(-0.2,0.) rectangle (3.2,2.3);
\draw [line width=0.8pt] (0.,1.)-- (0.,0.);
\draw [line width=0.8pt] (0.,0.)-- (2.,0.);
\draw [line width=0.8pt] (1.,0.5)-- (1.,1.5);
\draw [line width=0.8pt] (2.,0.)-- (2.,1.);
\draw [line width=0.8pt] (3.,0.)-- (3.,1.);
\draw(0.,1.5) node {\tiny #1};
\draw(1.,2.) node {\tiny #2};
\draw(3.,1.5) node {\tiny #3};
\end{tikzpicture}}

\newcommand{\nciquatredouze}[3]{
\begin{tikzpicture}[line cap=round,line join=round,>=triangle 45,x=0.3cm,y=0.3cm]
\clip(-0.2,0.) rectangle (3.2,2.3);
\draw [line width=0.8pt] (0.,1.)-- (0.,0.);
\draw [line width=0.8pt] (1.,0.)-- (3.,0.);
\draw [line width=0.8pt] (1.,0.)-- (1.,1.);
\draw [line width=0.8pt] (2.,0.5)-- (2.,1.5);
\draw [line width=0.8pt] (3.,0.)-- (3.,1.);
\draw(0.,1.5) node {\tiny #1};
\draw(1.,1.5) node {\tiny #2};
\draw(2.,2.) node {\tiny #3};
\end{tikzpicture}}

\newcommand{\nciquatretreize}[3]{
\begin{tikzpicture}[line cap=round,line join=round,>=triangle 45,x=0.3cm,y=0.3cm]
\clip(-0.2,0.) rectangle (3.2,2.3);
\draw [line width=0.8pt] (0.,1.)-- (0.,0.);
\draw [line width=0.8pt] (0.,0.)-- (3.,0.);
\draw [line width=0.8pt] (1.,0.5)-- (1.,1.5);
\draw [line width=0.8pt] (2.,0.5)-- (2.,1.5);
\draw [line width=0.8pt] (3.,0.)-- (3.,1.);
\draw(0.,1.5) node {\tiny #1};
\draw(1.,2.) node {\tiny #2};
\draw(2.,2.) node {\tiny #3};
\end{tikzpicture}}

\newcommand{\ncex}{
\begin{tikzpicture}[line cap=round,line join=round,>=triangle 45,x=0.3cm,y=0.3cm]
\clip(-0.2,0.) rectangle (14.2,2.8);
\draw [line width=0.8pt] (0.,0.)-- (0.,1.);
\draw [line width=0.8pt] (1.,0.5)-- (1.,1.5);
\draw [line width=0.8pt] (2.,1.)-- (2.,2.);
\draw [line width=0.8pt] (3.,0.5)-- (3.,1.5);
\draw [line width=0.8pt] (4.,0.)-- (4.,1.);
\draw [line width=0.8pt] (5.,0.5)-- (5.,1.5);
\draw [line width=0.8pt] (6.,0.5)-- (6.,1.5);
\draw [line width=0.8pt] (7.,0.5)-- (7.,1.5);
\draw [line width=0.8pt] (8.,0.)-- (8.,1.);
\draw [line width=0.8pt] (9.,0.5)-- (9.,1.5);
\draw [line width=0.8pt] (10.,0.5)-- (10.,1.5);
\draw [line width=0.8pt] (11.,0.)-- (11.,1.);
\draw [line width=0.8pt] (12.,0.)-- (12.,1.);
\draw [line width=0.8pt] (13.,0.)-- (13.,1.);
\draw [line width=0.8pt] (14.,0.)-- (14.,1.);
\draw [line width=0.8pt] (0.,0.)-- (12.,0.);
\draw [line width=0.8pt] (13.,0.)-- (14.,0.);
\draw [line width=0.8pt] (1.,0.5)-- (3.,0.5);
\draw [line width=0.8pt] (5.,0.5)-- (7.,0.5);
\draw(0.,1.5) node {\tiny 1};
\draw(1.,2) node {\tiny 2};
\draw(2.,2.5) node {\tiny 3};
\draw(5.,2.) node {\tiny 4};
\draw(9.,2.) node {\tiny 5};
\draw(10.,2.) node {\tiny 6};
\draw(13.,1.5) node {\tiny 7};
\end{tikzpicture}}

\newcommand{\ncexdeux}{
\begin{tikzpicture}[line cap=round,line join=round,>=triangle 45,x=0.3cm,y=0.3cm]
\clip(-0.2,0.) rectangle (6.2,1.8);
\draw [line width=0.8pt] (0.,0.)-- (0.,1.);
\draw [line width=0.8pt] (1.,0.)-- (1.,1.);
\draw [line width=0.8pt] (2.,0.)-- (2.,1.);
\draw [line width=0.8pt] (3.,0.)-- (3.,1.);
\draw [line width=0.8pt] (4.,0.)-- (4.,1.);
\draw [line width=0.8pt] (5.,0.)-- (5.,1.);
\draw [line width=0.8pt] (6.,0.)-- (6.,1.);
\draw [line width=0.8pt] (0.,0.)-- (4.,0.);
\draw [line width=0.8pt] (5.,0.)-- (6.,0.);
\draw(0.,1.5) node {\tiny 1};
\draw(5.,1.5) node {\tiny 7};
\end{tikzpicture}}

\newcommand{\ncextrois}{
\begin{tikzpicture}[line cap=round,line join=round,>=triangle 45,x=0.3cm,y=0.3cm]
\clip(-0.2,0.) rectangle (15.2,2.8);
\draw [line width=0.8pt] (0.,0.)-- (0.,1.);
\draw [line width=0.8pt] (1.,0.5)-- (1.,1.5);
\draw [line width=0.8pt] (2.,1.)-- (2.,2.);
\draw [line width=0.8pt] (3.,0.5)-- (3.,1.5);
\draw [line width=0.8pt] (4.,0.)-- (4.,1.);
\draw [line width=0.8pt] (5.,0.5)-- (5.,1.5);
\draw [line width=0.8pt] (6.,0.5)-- (6.,1.5);
\draw [line width=0.8pt] (7.,0.5)-- (7.,1.5);
\draw [line width=0.8pt] (8.,0.)-- (8.,1.);
\draw [line width=0.8pt] (9.,0.5)-- (9.,1.5);
\draw [line width=0.8pt] (10.,0.5)-- (10.,1.5);
\draw [line width=0.8pt] (11.,0.)-- (11.,1.);
\draw [line width=0.8pt] (12.,0.)-- (12.,1.);
\draw [line width=0.8pt] (13.,0.)-- (13.,1.);
\draw [line width=0.8pt] (14.,0.)-- (14.,1.);
\draw [line width=0.8pt] (15.,0.)-- (15.,1.);
\draw [line width=0.8pt] (0.,0.)-- (12.,0.);
\draw [line width=0.8pt] (13.,0.)-- (14.,0.);
\draw [line width=0.8pt] (1.,0.5)-- (3.,0.5);
\draw [line width=0.8pt] (5.,0.5)-- (7.,0.5);
\draw(0.,1.5) node {\tiny 1};
\draw(1.,2) node {\tiny 2};
\draw(2.,2.5) node {\tiny 3};
\draw(5.,2.) node {\tiny 4};
\draw(9.,2.) node {\tiny 5};
\draw(10.,2.) node {\tiny 6};
\draw(13.,1.5) node {\tiny 7};
\draw(15.,1.5) node {\tiny 8};
\end{tikzpicture}}

\theoremstyle{plain}
\newtheorem{theo}{Theorem}[section]
\newtheorem{lemma}[theo]{Lemma}
\newtheorem{cor}[theo]{Corollary}
\newtheorem{prop}[theo]{Proposition}
\newtheorem{defi}[theo]{Definition}

\theoremstyle{remark}
\newtheorem{remark}{Remark}[section]
\newtheorem{notation}{Notations}[section]
\newtheorem{example}{Example}[section]

\newcommand{\rond}[1]{*++[o][F-]{#1}}
\newcommand{\K}{\mathbb{K}}
\newcommand{\N}{\mathbb{N}}
\newcommand{\bfG}{\mathbf{G}}
\newcommand{\id}{\mathrm{Id}}
\newcommand{\com}{\mathbf{Com}}
\newcommand{\bfP}{\mathbf{P}}
\newcommand{\bfQ}{\mathbf{Q}}
\newcommand{\eq}{\mathcal{E}}
\newcommand{\cl}{\mathrm{cl}}
\newcommand{\sym}{\mathfrak{S}}
\newcommand{\Char}{\mathrm{Char}}
\newcommand{\tdelta}{\tilde{\Delta}}

\newcommand{\vect}{\mathrm{Vect}}
\renewcommand{\ker}{\mathrm{Ker}}

\newcommand{\calF}{\mathcal{F}}
\newcommand{\calH}{\mathcal{H}}
\newcommand{\bfH}{\mathbf{H}}
\newcommand{\ncp}{\mathcal{NCP}}
\newcommand{\bfncp}{\mathbf{NCP}}
\newcommand{\cat}{\mathrm{cat}}
\newcommand{\coinv}{\mathrm{coInv}}
\newcommand{\bl}{\mathrm{bl}}
\newcommand{\base}{\mathrm{Base}}
\newcommand{\ext}{\mathrm{le}}
\newcommand{\alg}{\K[\ncp]}
\newcommand{\Alg}{\com \circ \bfncp}
\newcommand{\calP}{\mathcal{P}}
\newcommand{\set}{\mathbf{Set}}
\newcommand{\Vect}{\mathbf{Vect}}
\newcommand{\bloc}{\mathrm{Bl}}
\newcommand{\lin}{\mathrm{Lin}}
\newcommand{\linstr}{\lin^s}
\newcommand{\GFDB}{G_{\mathrm{FdB}}}
\newcommand{\HFDB}{H_{\mathrm{FdB}}}
\newcommand{\phiFDB}{\Phi_{\mathrm{FdB}}}

\newcommand{\parti}{\mathrm{Part}}

\begin{document}

\title{Cointeraction on noncrossing partitions and related polynomial invariants}
\date{}
\author{Lo\"ic Foissy}
\affil{\small{Univ. Littoral Côte d'Opale, UR 2597
LMPA, Laboratoire de Mathématiques Pures et Appliquées Joseph Liouville
F-62100 Calais, France}.\\ Email: \texttt{foissy@univ-littoral.fr}}

\maketitle

\begin{abstract}
We study the structure of two cointeracting bialgebras on noncrossing partitions appearing in the theory of free probability. 
The first coproduct is given by separation of the blocks of the partitions into two parts, with respect to the nestings, while the second one is given by fusion of blocks. 
This structure implies the existence of a unique polynomial invariant $\phi_\ncp$ respecting the product and both coproducts.
We give a combinatorial interpretation of $\phi_\ncp$, study its values at -1 and use it for the computation of the antipode. 
We also give several results on its coefficients when applied to noncrossing partitions with no nesting. 
This leads to unexpected links with harmonic nested sums, Riordan arrays, composition of formal series and generalized Stirling numbers. 
This polynomial invariant is shown to be related to other ones, counting increasing or strictly increasing maps for the nesting order on noncrossing partitions, through the action of several characters.
\end{abstract}

\\

\AMSclass{16T05 16T30 05A17}

\tableofcontents

\section*{Introduction}

Noncrossing partitions have a rich combinatorics which has been used for example in the theory of free probability \cite{Biane2002,Biane1997,Nica2006,Simion2000}. 
They can be given several algebraic structures which are used to study relations between cumulants and moments in the theory of free probability \cite{Celestino2021,Celestino2022}.
In the present article, we study the double bialgebraic structure on noncrossing partitions, already defined in \cite{Foissy38} with operadic tools, and its applications to polynomial invariants. 
The presentation is based here on our construction of contraction-extraction coproducts \cite{Foissy41} in the context of species.

We shall denote by $\ncp$ the set of noncrossing partitions:
\[\ncp=\left\{\begin{array}{c}
\ncun;\: \ncdeuxun,\:\ncdeuxdeux;\:\nctroisun,\:\nctroisdeux,\:\nctroistrois,\:\nctroisquatre,\:\nctroiscinq;\:\\
\ncquatreun,\:\ncquatredeux,\:\ncquatretrois,\:\ncquatrequatre,\:\ncquatrecinq,\:\ncquatresix,\:\ncquatresept,\:\\
\ncquatrehuit,\:\ncquatreneuf,\:\ncquatredix,\:\ncquatreonze,\:\ncquatredouze,\:\ncquatretreize,\:\ncquatrequatorze,\ldots
\end{array}\right\}.\]
Graphically, noncrossing partitions will be represented with legs (vertical segments), united in blocks by horizontal segments.
The gradation which is used here is not given by the number of legs as usual, but rather by the number of blocks, see Definition \ref{defi2.1}. 
The object which will be considered all along the text is the polynomial algebra generated by $\ncp$, which we denote by $\alg$: a basis of this algebra is given by commutative monomials in noncrossing partitions.
In order to avoid confusion between these monomials and non connected noncrossing partitions, the product of $\alg$ is denoted by $\cdot$. 
For example, in $\alg$, $\ncun \cdot \ncdeuxun$, $\nctroistrois$ and $\nctroisdeux$ are pairwise different. 
The first coproduct, introduced in Proposition \ref{prop2.6}, is given by  separation of the blocks of a noncrossing partitions into two parts respecting the nesting order relation. For example,
  \begin{align*}
\Delta(\ncun)&=\ncun\otimes 1+1\otimes \ncun,\\
\Delta(\ncdeuxun)&=\ncdeuxun \otimes 1+1\otimes \ncdeuxun,\\
\Delta(\ncdeuxdeux)&=\ncdeuxdeux\otimes 1+1\otimes \ncdeuxdeux+2\ncun\otimes \ncun,\\
\Delta(\nctroisun)&=\nctroisun\otimes 1+1\otimes \nctroisun,\\
\Delta(\nctroisdeux)&=\nctroisdeux\otimes 1+1\otimes \nctroisdeux+\ncdeuxun\otimes \ncun+\ncun\otimes \ncdeuxun,\\
\Delta(\nctroistrois)&=\nctroistrois\otimes 1+1\otimes \nctroistrois+\ncdeuxun\otimes \ncun+\ncun\otimes \ncdeuxun,\\
\Delta(\nctroisquatre)&=\nctroisquatre\otimes 1+1\otimes \nctroisquatre+\ncdeuxun\otimes \ncun,\\
\Delta(\nctroiscinq)&=\nctroiscinq\otimes 1+1\otimes \nctroiscinq+\ncun\otimes (\ncun\cdot \ncun+2\ncdeuxdeux)
+3\ncdeuxdeux\otimes \ncun.
\end{align*}
This gives a Hopf algebra, graded in two different ways, firstly by the number of legs of noncrossing partitions, secondly by the number of blocks. 
In particular, the noncrossing partitions $J_n=\{\{1\},\ldots,\{n\}\}$ with only blocks of size 1 generates a Hopf subalgebra isomorphic to the Faà di Bruno Hopf algebra, which group of characters is the group of formal diffeomorphisms tangent to the identity (Proposition \ref{prop2.8}).
Using the formalism defined in \cite{Foissy41}, we give $\alg$ a second coproduct, given by fusion of blocks (Proposition \ref{prop2.10}). For example,
\begin{align*}
\delta(\ncun)&=\ncun\otimes \ncun,\\
\delta(\ncdeuxun)&=\ncdeuxun\otimes \ncdeuxun,\\
\delta(\ncdeuxdeux)&=\ncdeuxdeux\otimes \ncun\cdot \ncun+\ncdeuxun\otimes \ncdeuxdeux,\\
\delta(\nctroisun)&=\nctroisun\otimes \nctroisun,\\
\delta(\nctroisdeux)&=\nctroisdeux\otimes \ncdeuxun\cdot\ncun+\nctroisun\otimes \nctroisdeux,\\
\delta(\nctroistrois)&=\nctroistrois\otimes \ncdeuxun\cdot\ncun+\nctroisun\otimes \nctroistrois,\\
\delta(\nctroisquatre)&=\nctroisquatre\otimes \ncdeuxun\cdot\ncun+\nctroisun\otimes \nctroisquatre,\\
\delta(\nctroiscinq)&=\nctroiscinq\otimes \ncun \cdot \ncun\cdot\ncun
+(\nctroisdeux+\nctroistrois+\nctroisquatre)\otimes \ncun\cdot \ncdeuxdeux+\nctroisun\otimes \nctroiscinq.
\end{align*}
These two coproducts were also studied in \cite{Foissy38}, with applications to free probabilities.
We obtained that $(\alg,m,\Delta,\delta)$ is a bialgebra in cointeraction (we shall here call these objects double bialgebras), that is to say that the bialgebra $(\alg,m,\Delta)$
is a bialgebra in the category of right comodules over the bialgebra $(\alg,m,\delta)$, with the coaction given by $\delta$ itself. The most remarkable property that this implies is
\[(\Delta \otimes \id)\circ \delta=m_{1,3,24}\circ (\delta \otimes \delta)\circ \Delta,\]
where $m_{1,3,24}:\alg^{\otimes 4}\longrightarrow \alg^{\otimes 3}$ send $a_1\otimes a_2\otimes a_3\otimes a_4$ to $a_1\otimes a_3\otimes a_2\cdot a_4$. \\

Let us give a few reminders on the results we obtained on double bialgebras, see \cite{Foissy37,Foissy36} for a detailed exposition. More detailed can be found in the first section of this text.
Recall first that the polynomial algebra $\K[X]$ is also a double bialgebra, with its multiplicative coproducts defined by
\begin{align*}
\Delta(X)&=X\otimes 1+1\otimes X,&\delta(X)&=X\otimes X.
\end{align*}
A polynomial invariant is a bialgebra morphism from $(\alg,m,\Delta)$ to $(\K[X],m,\Delta)$. 
Then, there exists a unique polynomial invariant $\phi_\ncp$ which is also a bialgebra morphism from $(\alg,m,\delta)$ to $(\K[X],m,\delta)$. If $\pi$ is a noncrossing partition,
\[\phi_\ncp(\pi)=\sum_{k=0}^\infty \epsilon_\delta^{\otimes n} \circ \tdelta^{(n-1)}(\pi) \dfrac{X(X-1)\ldots (X-n+1)}{n!},\]
where $\epsilon_\delta$ is the counit of $\delta$ and $\tdelta^{(n-1)}$ are the iterated reduced coproducts associated to $\Delta$.
We give a combinatorial interpretation of $\phi_\ncp(\pi)$ in Proposition \ref{prop3.1}: for any integer $N$, $\phi_\ncp(\pi)(N)$ counts the number of $N$-valid colorations of $\pi$, 
that is to say maps $f$ from $\pi$ to $\{1,\ldots,N\}$ such that the two following conditions are satisfied:
\begin{itemize}
\item If $b,b'$ are two blocks of $\pi$ such that $b'$ is nested into $b$, then $f(b)<f(b')$.
\item If $b$ and $b''$ are two different blocks of $\pi$ such that $f(b)=f(b'')$ and $\max(b)<\min(b'')$, then there exists a block $b'$ of $\pi$ intersecting $]\max(b),\min(b'')[$  such that $f(b')<f(b)=f(b'')$. 
\end{itemize}
The second condition will be called the $b$-condition. We give in Proposition \ref{prop3.2} an inductive way to compute this chromatic polynomial $\phi_\ncp(\pi)$. We obtain, for example,
\begin{align*}
\phi_\ncp(\nctroisun)&=X,&\phi_\ncp(\nctroisdeux)&=X(X-1),\\
\phi_\ncp(\nctroistrois)&=X(X-1),&\phi_\ncp(\nctroisquatre)&=\dfrac{X(X-1)}{2},\\
\phi_\ncp(\nctroiscinq)&=X(X-1)\left(X-\dfrac{3}{2}\right).
\end{align*}
To this polynomial invariant is attached a character $\mu_\ncp$, defined by
\[\mu_\ncp(x)=\phi_\ncp(x)(-1),\]
for any $x\in \alg$. The antipode of $(\alg,m,\Delta)$ is then given by 
\[S=(\mu_\ncp\otimes \id)\circ \delta.\]
We show that this character can be explicitly computed (Proposition \ref{prop3.5}), and that it sends any noncrossing partition to a  signed product of Catalan numbers. 
As a consequence, we deduce some formulas for the coefficients of the inverse (for the composition) of a formal series (Proposition \ref{prop3.6}).

We then turn to the computation $\phi_\ncp(\pi)$, when $\pi$ has no nesting. In this case, $\phi_\ncp(\pi)$ depends only on the number $n$ of blocks of $\pi$.
We decompose these polynomials, first in the basis of Hilbert polynomials $\left(\dfrac{X(X-1)\ldots (X-n+1)}{n!}\right)_{n\geq 0}$, then in the basis of monomials $(X^n)_{n\geq 0}$. 
In the first case, their coefficients can be  inductively computed (Proposition \ref{prop3.8}). Closed formulas for certain of these coefficients are given in Corollary \ref{cor3.10}, with the help of multiple nested harmonic sums. 
In the second case, Proposition \ref{prop3.11} shows that these coefficients are related to the exponentiation of a particular infinite matrix, namely the Riordan array of $(1+X,X(1+X))$. 
We also give a combinatorial interpretation of the coefficient of $X^n$ and $X^{n-1}$ in $\phi_\ncp(\pi)$ when $\pi$ has $n$ blocks in Corollaries \ref{cor3.17} and \ref{cor3.19}, 
using a particular invertible character $\lambda_\ncp$ of $(\alg,m,\delta)$. 
This character is also related to the infinitesimal generator of $X(1+X)$, see Proposition \ref{prop3.20}, and to generalized Stirling numbers, see Example \ref{ex3.5}. 

We introduce two other bialgebra morphisms $\Lambda$ and $\Lambda^s$ from $(\alg,m,\Delta)$ to $(\K[X],m,\Delta)$, counting colorations where the $b$-condition is abandoned and, in the $\Lambda$ case, 
the increasing condition is weakened (Theorem \ref{theo3.23}). For example,
\begin{align*}
\Lambda(\nctroisun)&=X,&\Lambda^s(\nctroisun)&=X,\\
\Lambda(\nctroisdeux)&=X^2,&\Lambda^s(\nctroisdeux)&=X^2,\\
\Lambda(\nctroistrois)&=X^2,&\Lambda^s(\nctroistrois)&=X^2,\\
\Lambda(\nctroisquatre)&=\dfrac{X(X+1)}{2},&\Lambda^s(\nctroisquatre)&=\dfrac{X(X-1)}{2},\\
\Lambda(\nctroiscinq)&=X^3,&\Lambda^s(\nctroiscinq)&=X^3.
\end{align*}
These two morphisms are related by a duality principle: for any noncrossing partition $\pi$,
\[\Lambda(\pi)(X)=(-1)^{|\pi|}\Lambda^s(\pi)(-X).\]
We also show in Corollary \ref{cor3.24} that they are related to $\phi_\ncp$ trough two simple characters $\lambda$ and $\lambda^s$, such that
\begin{align*}
\Lambda&=(\phi_\ncp \otimes \lambda)\circ \delta,&\Lambda^s&=(\phi_\ncp \otimes \lambda^s)\circ \delta.
\end{align*}
The inverses $\mu$ and $\mu^s$ for the convolution associated to $\delta$ of these characters are also studied. Note that 
\begin{align*}
\phi_\ncp&=(\Lambda \otimes \mu)\circ \delta=(\Lambda^s\otimes \mu^s)\circ \delta.
\end{align*}
We give a simple formula for $\mu^s$ in Proposition \ref{prop3.27} and a more complicated way to compute $\mu$ in Proposition \ref{propr3.32}.

Finally, in the last section of this paper, we prove that there is no double bialgebra morphism from $\alg$ to the double bialgebras of hypergraphs \cite{Foissy44} or mixed graphs \cite{Foissy45}, 
sending a noncrossing partition to a sum of hypergraphs or to a mixed graph, showing in this way the specificity of the combinatorics of noncrossing partitions, see Proposition \ref{prop4.1}. \\

\textbf{Acknowledgments}. The author acknowledges support from the grant ANR-20-CE40-0007 \emph{Combinatoire Algébrique, Résurgence, Probabilités Libres et Opérades}. The author is very grateful to Thomas Copeland for pointing the link between the antipode and the coefficients appearing in the inversion of a formal series, which lead to a second version of this text with additions on the Faà di Bruno Hopf algebra. \\ 

\begin{notation} \begin{enumerate}
\item We denote by $\K$ a commutative field of characteristic zero. Any vector space in this text will be taken over $\K$.
\item For any $n\in \N$, we denote by $[n]$ the set $\{1,\ldots,n\}$. In particular, $[0]=\emptyset$.
\item If $(C,\Delta)$ is a (coassociative but not necessarily counitary) coalgebra, we denote by $\Delta^{(n)}$ the $n$-th iterated coproduct of $C$:
$\Delta^{(0)}=\id_C$, $\Delta^{(1)}=\Delta$ and if $n\geq 2$,
\[\Delta^{(n)}=\left(\Delta \otimes \id_C^{\otimes (n-1)}\right)\circ \Delta^{(n-1)}:C\longrightarrow C^{\otimes (n+1)}.\]
\item If $(B,m,\Delta)$ is a bialgebra of unit $1_B$ and of counit $\varepsilon_B$, let us denote by $B_+=\ker(\varepsilon_B)$ its augmentation ideal. We define a coproduct on $B_+$ by
\begin{align*}
&\forall x\in B_+,&\tdelta(x)=\Delta(x)-x\otimes 1_B-1_B\otimes x.
\end{align*}
Then $(B_+,\tdelta)$ is a coassociative (generally not counitary) coalgebra. 
In consequence, we shall be able to consider the iterated reduced coproducts $\tdelta^{(n-1)}$, with $n\geq 1$.
\item For any $n\geq 1$, we denote by $H_n(X)$ the $n$-th Hilbert polynomial:
\[H_n(X)=\frac{X(X-1)\ldots (X-n+1)}{n!} \in \K[X].\]
By convention, $H_0=1$. 
\end{enumerate}\end{notation}

\section{Reminders}

\subsection{On species}

We denote by $\set$ the category of finite sets with bijections and by $\Vect$ the category of vector spaces.
A (linear) species is a functor $\bfP$ from $\set$ to $\Vect$. 
Let $\bfP$ be a species. We fix the notations:
\begin{itemize}
\item For any finite set $X$, the vector space associated to $X$ by $\bfP$ is denoted by $\bfP[X]$.
\item For any bijection $\sigma:X\longrightarrow Y$  between two finite sets, the linear map associated to $\sigma$ by $\bfP$ is denoted by $\bfP[\sigma]:\bfP[X]\longrightarrow \bfP[Y]$. 
\end{itemize}
Note that for any finite set $X$, $\bfP[\id_X]=\id_{\bfP[X]}$ and that for any bijections $\sigma:X\longrightarrow Y$
and $\tau:Y\longrightarrow Z$ between finite sets, $\bfP(\tau\circ \sigma)=\bfP(\tau)\circ \bfP(\sigma)$.\\

Let $\bfP$ and $\bfQ$ be two species. A morphism between $\bfP$ and $\bfQ$ is a natural transformation between the two functors $\bfP$ and $\bfQ$, that is to say, 
for any finite set $X$, a linear map $f_X:\bfP[X]\longrightarrow \bfP[Y]$ such that for any bijection $\sigma:X\longrightarrow Y$ between two finite sets, the following diagram commutes:
\[\xymatrix{\bfP[X]\ar[r]^{\bfP[\sigma]}\ar[d]_{f_X}&\bfP[Y]\ar[d]^{f_Y}\\
\bfQ[X]\ar[r]_{\bfQ[\sigma]}&\bfQ[Y]}\]

Species form a symmetric monoidal category, with the Cauchy tensor product $\otimes$: if $\bfP$ and $\bfQ$ are two species, for any finite set $X$,
\[\bfP\otimes \bfQ[X]=\bigoplus_{X=Y\sqcup Z} \bfP[Y]\otimes \bfQ[Z],\]
and if $\sigma:X\longrightarrow Y$  is a bijection between two finite sets, then 
\[\bfP\otimes \bfQ[\sigma]=\bigoplus_{X=Y\sqcup Z} \bfP[\sigma_{\mid Y}]\otimes \bfQ[\sigma_{\mid Z}].\]
For any species $\bfP$ and $\bfQ$, the flip $c_{\bfP,\bfQ}:\bfP\otimes \bfQ\longrightarrow \bfQ\otimes \bfP$ 
is defined by the following: for any pair $(X,Y)$ of disjoint sets,
\begin{align*}
c_{\bfP,\bfQ}&:\left\{\begin{array}{rcl}
\bfP[X]\otimes \bfQ[Y]&\longrightarrow&\bfQ[Y]\otimes \bfP[X]\\
x\otimes y&\longmapsto&y\otimes x.
\end{array}\right.\end{align*}

A twisted algebra (resp. coalgebra, bialgebra) is an algebra (resp. coalgebra, bialgebra) in the symmetric monoidal category of species with the Cauchy tensor product. Let us now give more details.

 A twisted algebra is a pair $(\bfP,m)$ where $\bfP$ is a species and, for any pair $(X,Y)$ of disjoint finite sets,
$m_{X,Y}:\bfP[X]\otimes \bfP[Y]\longrightarrow \bfP[X\sqcup Y]$ is a linear map, with the following properties:
\begin{itemize}
\item If $(X,Y)$ and $(X',Y')$ are two pairs of disjoint finite sets, and if  $\sigma:X\longrightarrow X'$ and $\tau:Y\longrightarrow Y'$ are bijections, then the following diagram commutes:
\[\xymatrix{\bfP[X]\otimes \bfP[Y]\ar[r]^{m_{X,Y}}\ar[d]_{\bfP[\sigma]\otimes \bfP[\tau]}&\bfP[X\sqcup Y]\ar[d]^{\bfP[\sigma \sqcup \tau]}\\
\bfP[X']\otimes \bfP[Y']\ar[r]_(.55){m_{X',Y'}}&\bfP[X'\sqcup Y']}\]
\item The product $m$ is associative: for any triple $(X,Y,Z)$ of pairwise disjoint sets,
\[m_{X\sqcup Y,Z}\circ (m_{X,Y}\otimes \id_{\bfP[Z]})=m_{X,Y\sqcup Z}\circ (\id_{\bfP[X]}\otimes m_{Y,Z}).\]
\item The product $m$ has a unit $1_\bfP\in \bfP[\emptyset]$: for any finite set $X$, for any $x\in \bfP[X]$,
\[m_{\emptyset,X}(1_\bfP\otimes x)=m_{X,\emptyset}(x\otimes 1_\bfP)=x.\]
\end{itemize}

A twisted coalgebra is a pair $(P,\Delta)$ where $\bfP$ is a species and for any pair $(X,Y)$ of finite sets, 
$\Delta_{X,Y}:\bfP[X\sqcup Y]\longrightarrow \bfP[X]\otimes \bfP[Y]$ is a linear map, with the following properties:
\begin{itemize}
\item  If $(X,Y)$ and $(X',Y')$ are two pairs of disjoint finite sets, and if $\sigma:X\longrightarrow X'$ and $\tau:Y\longrightarrow Y'$ are bijections, then the following diagram commutes:
\[\xymatrix{\bfP[X\sqcup Y]\ar[r]^{\Delta_{X,Y}}\ar[d]_{\bfP[\sigma\sqcup \tau]}&\bfP[X]\otimes \bfP[Y]\ar[d]^{\bfP[\sigma]\otimes \bfP[\tau]}\\
\bfP[X'\sqcup Y']\ar[r]_{\Delta_{X',Y'}}&\bfP[X']\otimes \bfP[Y']}\]
\item The coproduct $\Delta$ is coassociative: for any triple $(X,Y,Z)$ of pairwise disjoint finite sets,
\[(\Delta_{X,Y}\otimes \id_{\bfP[Z]})\circ \Delta_{X\sqcup Y,Z}=(\id_{\bfP[Z]}\otimes \Delta_{Y,Z})\circ \Delta_{X,Y\sqcup Z}.\]
\item The coproduct $\Delta$ has a counit $\varepsilon_\Delta \in \bfP[\emptyset]^*$: for any finite set $X$,
\[(\varepsilon_\Delta \otimes \id_{\bfP[X]})\circ \Delta_{\emptyset,X}
=(\id_{\bfP[X]}\otimes \varepsilon_\Delta)\circ \Delta_{X,\emptyset}=\id_{\bfP[X]}.\]
\end{itemize}

A twisted bialgebra is a triple $(\bfP,m,\Delta)$ such that:
\begin{itemize}
\item $(\bfP,m)$ is a twisted algebra. Its unit is denoted by $1_\bfP$.
\item $(\bfP,\Delta)$ is a twisted coalgebra. Its counit is denoted by $\varepsilon_\Delta$.
\item For any pairs $(X_1,X_2)$ and $(Y_1,Y_2)$ of disjoint finite sets, such that $X_1\sqcup X_2=Y_1\sqcup Y_2$,
\begin{align*}
\Delta_{Y_1,Y_2}\circ m_{X_1,X_2}&=(m_{X_1\cap Y_1,X_2\cap Y_1}\otimes m_{X_1\cap Y_2,X_2\cap Y_2})\\
&\circ (\id_{\bfP[X_1\cap Y_1]}\otimes c_{\bfP[X_1\cap Y_2,X_2\cap Y_1]}\otimes \id_{\bfP[X_2\cap Y_2]})\\
&\circ(\Delta_{X_1\cap Y_1,X_1\cap Y_2}\otimes \Delta_{X_2\cap Y_1,X_2\cap Y_2}).
\end{align*}
\item For any $x,y\in \bfP[\emptyset]$, $\varepsilon_\Delta(xy)=\varepsilon_\Delta(x)\varepsilon_\Delta(y)$.
\item $\Delta_{\emptyset,\emptyset}(1_\bfP)=1_\bfP\otimes 1_\bfP$ and $\varepsilon_\Delta(1_\bfP)=1$.\\
\end{itemize}

 The bosonic Fock functor \cite{Aguiar2010}  sends any species $\bfP$ to
\[\calF[\bfP]=\bigoplus_{n=0}^\infty \coinv(\bfP[n])=\bigoplus_{n=0}^\infty \frac{\bfP[n]}{\vect(\bfP[\sigma](p)-p\mid \sigma\in \sym_n,\: p\in \bfP[n])}.\]
This is a functor of symmetric monoidal categories from species to (graded) vector spaces.
Therefore, if $\bfP$ is a twisted algebra (resp. coalgebra, bialgebra), then $\calF[\bfP]$ is a (graded) algebra
(resp. coalgebra, bialgebra).

\subsection{On double bialgebras}

We refer to \cite{Foissy37,Foissy36,Foissy40} for the details.

\begin{defi}
A double bialgebra is a family $(B,m,\Delta,\delta)$ such that:
\begin{enumerate}
\item $(B,m,\Delta)$ and $(B,m,\delta)$ are bialgebras. Their common unit is denoted by $1_B$.
The counits of $\Delta$ and $\delta$ are respectively denoted by $\varepsilon_\Delta$ and $\epsilon_\delta$.
\item $(B,m,\Delta)$ is a bialgebra in the category of right comodules over $(B,m,\delta)$, with the coaction $\delta$, seen as a coaction over itself. This is equivalent to the two following assertions:
\begin{align*}
&\forall x\in B,&(\varepsilon_\Delta \otimes \id_B)\circ \delta(x)&=\varepsilon_\Delta(x)1_B,\\
&&(\Delta \otimes \id_B)\circ \delta&=m_{1,3,24}\circ (\delta\otimes \delta)\circ \Delta,
\end{align*}  
where $m_{1,3,24}:B^{\otimes 4}\longrightarrow B^{\otimes 3}$ sends $x_1\otimes x_2\otimes x_3\otimes x_4$ to $x_1\otimes x_3\otimes x_2x_4$. 
Note that $m_{1,3,24}\circ (\delta \otimes \delta)$ is in fact the coaction of $B\otimes B$. 
\end{enumerate}\end{defi}

An example of double bialgebra is given by the usual polynomial algebra $\K[X]$, with its usual product $m$ and the two (multiplicative) coproducts defined by
\begin{align*}
\Delta(X)&=X\otimes 1+1\otimes X,&\delta(X)&=X\otimes X.
\end{align*}
The counits are given by
\begin{align*}
\varepsilon_\Delta:&\left\{\begin{array}{rcl}
\K[X]&\longrightarrow&\K\\
P(X)&\longmapsto&P(0),
\end{array}\right.&
\epsilon_\delta:&\left\{\begin{array}{rcl}
\K[X]&\longrightarrow&\K\\
P(X)&\longmapsto&P(1).
\end{array}\right.\end{align*}

\begin{prop}\label{prop1.2}
Let $(B,m,\Delta,\delta)$ be a double bialgebra. 
\begin{enumerate}
\item We denote by $\Char(B)$ the set of characters of $B$, that is to say the set of algebra morphisms from $B$ to $\K$. This sets inherits two associative and unitary products defined by
\begin{align*}
&\forall \lambda,\mu\in \Char(B),&\lambda*\mu&=(\lambda \otimes \mu)\circ \Delta,&\lambda\star \mu&=(\lambda \otimes \mu)\circ \delta.
\end{align*}
The units of the products $*$ and $\star$ are respectively $\varepsilon_\Delta$ and $\epsilon_\delta$.
\item Let $(A,m,\Delta)$ be a bialgebra. We denote by $M_{B\rightarrow A}$ the set of bialgebra morphisms from $(B,m,\Delta)$ to $(A,m,\Delta)$. 
Then the monoid $(\Char(B),\star)$ acts on $M_{B\rightarrow A}$ via the following right action:
\begin{align*}
\leftsquigarrow&:\left\{\begin{array}{rcl}
M_{B\rightarrow A}\times \Char(B)&\longrightarrow&M_{B\rightarrow A}\\
(\phi,\lambda)&\longmapsto&\phi\leftsquigarrow\lambda=(\phi\otimes \lambda)\circ \delta.
\end{array}\right.
\end{align*}\end{enumerate}\end{prop}

The double structure allows to find the antipode for the first structure, whenever it exists:

\begin{theo}\label{theo1.3}  \cite[Corollary 2.3]{Foissy40}
Let $(B,m,\Delta,\delta)$ be a double bialgebra. 
\begin{enumerate}
\item Then $(B,m,\Delta)$ is a Hopf algebra if, and only if, the character $\epsilon_\delta$ has an inverse $\mu_B$ for the convolution product $*$ dual to $\Delta$. 
Moreover, if this holds, the antipode of $(B,m,\Delta)$ is given by
\[S=(\mu_B \otimes \id_B)\circ \delta.\]
\item Let $\phi_B:B\longrightarrow \K[X]$ be a double bialgebra morphism. Then $\epsilon_\delta$ has an inverse for the convolution product $*$, given by
\begin{align*}
\mu_B&: \left\{\begin{array}{rcl}
B&\longrightarrow&\K\\
x&\longmapsto&\phi_B(x)(-1).
\end{array}\right.
\end{align*}\end{enumerate}\end{theo}

A double bialgebra $(B,m,\Delta,\delta)$ is connected if the reduced coproduct $\tdelta$ is locally nilpotent:
in other words, for any $x\in \ker(\varepsilon_\Delta)$, there exists $n\geq 1$ such that $\tdelta^{(n)}(x)=0$. If so, we obtain more results:

\begin{theo}\label{theo1.4}\cite[Theorem 3.9, Corollary 3.12]{Foissy40}
Let $(B,m,\Delta,\delta)$ be a connected double bialgebra.
\begin{enumerate}
\item There exists a unique double bialgebra morphism $\phi_B$ from $B$ to $\K[X]$. Moreover,
\begin{align*}
&\forall x\in \ker(\varepsilon_B),&\phi_B(x)&=\sum_{k=1}^\infty \epsilon_\delta^{\otimes k}\circ \tdelta^{(k-1)}(x) H_k(X).
\end{align*}
\item The two following maps are bijective, inverse one from the other:
\begin{align*}
&\left\{\begin{array}{rcl}
\Char(B)&\longrightarrow&M_{B\rightarrow \K[X]}\\
\lambda&\longmapsto&\phi_B\leftsquigarrow \lambda,
\end{array}\right.
&&\left\{\begin{array}{rcl}
M_{B\rightarrow \K[X]}&\longrightarrow&\Char(B)\\
\phi&\longmapsto&\left\{\begin{array}{rcl}
B&\longrightarrow&\K\\
x&\longmapsto&\phi(x)(1).
\end{array}\right.
\end{array}\right.
\end{align*}
\item For any $\lambda \in \Char(B)$, for any $x\in \ker(\varepsilon_\Delta)$, 
\[\phi_B\leftsquigarrow\lambda(x)=\sum_{k=1}^\infty \lambda^{\otimes k}\circ \tdelta^{(k-1)}(x) H_k(X).\]
\end{enumerate}\end{theo}

\section{Bialgebraic structures on noncrossing partitions}

\subsection{Reminders on noncrossing partitions}

\begin{defi}\label{defi2.1}
A noncrossing partition is a partition $\pi$ of a set $[n]$ with $n\in \N$, such that 
\begin{align*}
&\forall b\neq b'\in \pi,\: \forall x,z\in b,\: \forall y,t\in b',&x<y<z<t\mbox{ does not hold}.
\end{align*}
The elements of $\pi$ are called its blocks. The number of blocks of a noncrossing partition $\pi$ is denoted by $|\pi|$. The set of noncrossing partitions with $k$ blocks is denoted by $\ncp_k$. 
\end{defi}

We shall represent noncrossing partitions by diagrams. Elements of $[n]$ will be represented by vertical segments, from left to right, related by horizontal segments corresponding to the blocks. 
For example, the noncrossing partition $\{\{1,3\},\{2\},\{4\}\}$ is represented by $\ncquatreonze$.\\

Note that for any $k\geqslant 1$, $\ncp_k$ is an infinite set. For example,
\begin{align*}
\ncp_1&=\{\ncun;\:\ncdeuxun;\:\nctroisun;\:\ncquatreun;\:\ldots\},\\
\ncp_2&=\{\ncdeuxdeux;\:\nctroisdeux,\:\nctroistrois; \:\nctroisquatre,\:\ncquatredeux,\:\ncquatretrois,\:\ncquatrequatre,\:\ncquatrecinq,\:\ncquatresix,\:\ncquatresept;\:\ldots\},\\
\ncp_3&=\{\nctroiscinq;\:\ncquatrehuit,\:\ncquatreneuf,\:\ncquatredix,\:\ncquatreonze,\:\ncquatredouze,\:\ncquatretreize;\:\ldots\},\\
\ncp_4&=\{\ncquatrequatorze;\:\ldots\}.
\end{align*}
By convention, $\ncp_0=\emptyset$.

\begin{defi}
Let $X$ be a finite set. We denote by $\ncp[X]$  the set of bijections $f$ between an element $\pi$ of $\ncp_{|X|}$ and $X$.
If $\sigma:X\longrightarrow Y$ is a bijection between two finite sets and $f:\pi\longrightarrow X$ is an element of $\ncp[X]$, we put $\ncp[\sigma][f]=\sigma\circ f:\pi\longrightarrow Y$. 
This defines a species $\ncp$, which linearization is denoted by $\bfncp$: for any finite set $X$,
\[\bfncp[X]=\mathrm{Vect}(\ncp[X]).\]
\end{defi}

\begin{notation}
the elements of $\ncp[X]$ can be described as noncrossing partitions which blocks are indexed by the set $X$. Such an object$f:\pi\longrightarrow X$  will be denoted by $\pi=(\pi_x)_{x\in X}$, 
where for any $x\in X$, $\pi_x=f^{-1}(x)$. We shall represent graphically the elements of $\ncp[X]$, the index of the blocks being attached to their leftmost leg.
\end{notation}

\begin{example}\begin{align*}
\ncp[1]&=\left\{\nciun{1},\:\ncideuxun{1},\:\ncitroisun{1},\:\nciquatreun{1},\ldots\right\},\\
\ncp[2]&=\left\{\begin{array}{c}
\ncideuxdeux{1}{2},\:\ncideuxdeux{2}{1},\:\ncitroisdeux{1}{2},\:\ncitroisdeux{2}{1},\:\ncitroistrois{1}{2},\:\ncitroistrois{2}{1},\:\ncitroisquatre{1}{2},\:\ncitroisquatre{2}{1},\\
\nciquatresept{1}{2},\:\nciquatresept{2}{1},\:\nciquatredeux{1}{2},\:\nciquatredeux{2}{1},\:\nciquatretrois{1}{2},\:\nciquatretrois{2}{1},\\\
\nciquatrequatre{1}{2},\:\nciquatrequatre{2}{1},\:\nciquatrecinq{1}{2},\:\nciquatrecinq{2}{1},\:\nciquatresix{1}{2},\:\nciquatresix{2}{1},\ldots
\end{array}\right\},\\
\ncp[3]&=\left\{\begin{array}{c}
\ncitroiscinq{1}{2}{3},\:\ncitroiscinq{1}{3}{2},\:\ncitroiscinq{2}{1}{3},\:\ncitroiscinq{2}{3}{1},\:\ncitroiscinq{3}{1}{2},\:\ncitroiscinq{3}{2}{1},\:\\
\nciquatrehuit{1}{2}{3},\:\nciquatrehuit{1}{3}{2},\:\nciquatrehuit{2}{1}{3},\:\nciquatrehuit{2}{3}{1},\:\nciquatrehuit{3}{1}{2},\:\nciquatrehuit{3}{2}{1},\:\\
\nciquatreneuf{1}{2}{3},\:\nciquatreneuf{1}{3}{2},\:\nciquatreneuf{2}{1}{3},\:\nciquatreneuf{2}{3}{1},\:\nciquatreneuf{3}{1}{2},\:\nciquatreneuf{3}{2}{1},\:\\
\nciquatredix{1}{2}{3},\:\nciquatredix{1}{3}{2},\:\nciquatredix{2}{1}{3},\:\nciquatredix{2}{3}{1},\:\nciquatredix{2}{3}{1},\:\nciquatredix{3}{1}{2},\:\\
\nciquatreonze{1}{2}{3},\:\nciquatreonze{1}{3}{2},\:\nciquatreonze{2}{1}{3},\:\nciquatreonze{2}{3}{1},\:\nciquatreonze{3}{1}{2},\:\nciquatreonze{3}{2}{1},\:\\
\nciquatredouze{1}{2}{3},\:\nciquatredouze{1}{3}{2},\:\nciquatredouze{2}{1}{3},\:\nciquatredouze{2}{3}{1},\:\nciquatredouze{3}{1}{2},\:\nciquatredouze{3}{2}{1},\:\\
\nciquatretreize{1}{2}{3},\:\nciquatretreize{1}{3}{2},\:\nciquatretreize{2}{1}{3},\:\nciquatretreize{2}{3}{1},\:\nciquatretreize{3}{1}{2},\:\nciquatretreize{3}{2}{1},\ldots
\end{array}\right\}.\end{align*}\end{example}

\begin{notation}
We shall consider the species $\com \circ \bfncp$: for any finite set $X$, 
\[\Alg[X]=\bigoplus_{\sim \in \eq[X]} \left(\bigotimes_{I \in X/\sim}\bfncp[I]\right),\]
where $\eq[X]$ is the set of all equivalence relations on $X$. 
As $\com$ is a commutative twisted algebra and  the endofunctor $\calF_{\circ \bfncp}$ of the category of species is compatible with the Cauchy tensor product, $\Alg$ is naturally a commutative twisted algebra.
Its product is denoted by $\cdot$.  In other terms, a basis of $\Alg[X]$ is given by commutative monomials of noncrossing partitions $\pi_1\cdot \ldots \cdot \pi_k$, 
where for any $i$, $\pi_i$ is indexed by a nonempty set $I_i$, such that $I_1\sqcup \ldots \sqcup I_k=X$. The unit is the empty monomial $1\in \Alg[\emptyset]$. 
\end{notation}

\begin{example}
\begin{align*}
\Alg[1]&=\vect(\ncp[1]),\\
\Alg[2]&=\vect(\ncp[2])\oplus \vect\left(\nciun{1}\cdot \nciun{2},\:\nciun{1}\cdot \ncideuxun{2},\nciun{2}\cdot \ncideuxun{1},\: \ncideuxun{1}\cdot \ncideuxun{2},\ldots \right)
\end{align*}
Note that
\begin{align*}
\nciun{1}\cdot \ncideuxun{2}&=\ncideuxun{2}\cdot \nciun{1},&\nciun{1}\cdot \ncideuxun{2}&\neq \ncitroistrois{1}{2},&
\nciun{1}\cdot \ncideuxun{2}&\neq \nciun{2}\cdot \ncideuxun{1},&\ncideuxun{1}\cdot \nciun{2}&\neq \ncitroisdeux{1}{2}.
\end{align*}\end{example}

\subsection{The first coproduct}

Let us now define a coproduct on the twisted algebra $\com \circ \bfncp$. 

\begin{notation}
Let $\pi=(\pi_x)_{x\in X}$ be an $X$-indexed noncrossing partition and let $Y\subseteq X$. We put $\displaystyle I_Y=\bigcup_{y\in Y}\pi_y$ 
and we denote by $f_Y$ the unique increasing bijection from $I_Y$ to $[|I_Y|]$. The $Y$-indexed noncrossing partition $\pi_{\mid Y}$ is
\[\pi_{\mid Y}=(f_Y(\pi_y))_{y\in Y}.\]
Graphically, $\pi_Y$ is obtained by deleting the blocks of $\pi$ which are not indexed by an element of $Y$. By convention, $\pi_{\mid \emptyset}=1$. Note that $\pi_{\mid X}=\pi$.
\end{notation}

\begin{defi}\label{defi2.3}
Let $\pi=(\pi_x)_{x\in X}$ be an $X$-indexed noncrossing partition and $I\subseteq X$. We put $\displaystyle I'=\bigcup_{x\in I} \pi_x \subseteq[n]$, 
which we decompose into connected components $I'_1\sqcup \ldots \sqcup I'_k$, that is to say:
\begin{itemize}
\item For any $i\in [k]$, $I'_i$ is an interval of $[n]$.
\item For any $i\in [k-1]$, $\max(I'_i)<\min(I'_{i+1})-1$. 
\end{itemize}
We shall say that $I$ is an ideal of $\pi$ if for any $i\in [k]$, $I'_i$ is the union of blocks of $\pi$.  If $I$ is an ideal of $\pi$, we put
\[\pi_{\mid\cdot I}=\prod_{i\in [k]}^\cdot \pi_{\mid I'_i}.\]
\end{defi}

\begin{example}
Let us consider 
\[\pi=\ncex \in \ncp[7].\]
Then $\{2,3,4,5,6\}$ is an ideal of $\pi$. The associated connected components are $\{2,3,4\}$, $\{6,7,8\}$ and $\{10,11\}$, and
\begin{align*}
\pi_{\mid_\cdot \{2,3,4,5,6\}}&=\ncitroisquatre{2}{3}\cdot \ncitroisun{4}\cdot \ncideuxdeux{5}{6},&\pi_{\mid \{1,7\}}&=\ncexdeux.
\end{align*}\end{example}

Let us give a more usual characterization of the ideals of a noncrossing partitions, with the help of a classical order on the blocks.

\begin{defi}[\textbf{Nesting order}]
Let $\pi$ a noncrossing partition. We define a relation on $\pi$ by
\begin{align*}
&\forall b,b'\in \pi,&b\leq_\pi b'&\Longleftrightarrow b'\subseteq [\min(b),\max(b)].
\end{align*}
Then $\leq_\pi$ is a partial order on $\pi$. 
\end{defi}

\begin{lemma}
Let $\pi=(\pi_x)_{x\in X}$ be an $X$-indexed noncrossing partition and $I\subset X$. Then $I$ is an ideal of $\pi$ if, and only if, 
\begin{align*}
&\forall x,y\in X,&(x\in I\mbox{ and }\pi_x\leq_\pi \pi_y)&\Longrightarrow y\in I.
\end{align*}\end{lemma}

\begin{proof}
$\Longrightarrow$. Let $x,y\in X$, with $x\in I$ and $\pi_x\leq_\pi \pi_y$. As $I$ is an ideal of $\pi$, there exists a connected component $I'_i$ containing $\pi_x$. As $\pi_y\subseteq [\min(\pi_x),\max(\pi_y)]\subseteq I'_i$, $y\in I$. \\

$\Longleftarrow$.  Let $\pi_x$ be a block of $\pi$ intersecting $I'_i$. Let us assume that $\min(\pi_x)\notin I'_i$. As $J$ is the union of blocks of $\pi$, $\min(\pi_x)\in I'$, so belongs to a $I'_p$ with $p<i$. 
By definition of the connected components, there exists $k\in [n]\setminus J$, such that $\max(I'_p)<k<\min(I'_i)$. This element $k$ belongs to a block $\pi_y$ of $\pi$. 
Then $\min(\pi_x)<k\in \pi_y<\min(I'_i)\leqslant \max(\pi_x)$.
As $\pi$ is noncrossing, $\pi_x\leq_\pi \pi_y$. As $I$ is an ideal, $y\in I$, so $k\in I'$: this is a contradiction, so $\min(\pi_x)\in I'_i$. Similarly, $\max(\pi_y)\in I'_i$. As $I'_i$ is an interval, $\pi_x \subseteq I'_i$. 
Therefore, $I'_i$ is the union of the blocks of $\pi$ it intersects. 
\end{proof}

\begin{prop}\label{prop2.6}
Let $X=I\sqcup J$ be a finite set and $\pi \in \ncp[X]$. We define $\Delta_{I,J}(\pi)\in \bfncp[I]\otimes \Alg[J]$ by
\[\Delta(\pi)=\begin{cases}
\pi_{\mid I}\otimes \pi_{\mid_\cdot J} \mbox{ if $J$ is an ideal of $\pi$},\\
0\mbox{ otherwise.}
\end{cases}\]
We extend $\Delta$ as an algebra morphism from $\Alg$ to $(\Alg)^{\otimes 2}$. Then the triple $(\Alg,m,\Delta)$ is a twisted bialgebra. 
\end{prop}

\begin{proof}
By definition, $\Delta$ is an algebra morphism.  Let $\pi\in \ncp[X]$. As $\emptyset$ and $X$ are obviously ideals of $\pi$,
\begin{align*}
\Delta_{\emptyset,X}(\pi)&=1\otimes \pi,&\Delta_{X,\emptyset}(\pi)&=\pi\otimes 1,
\end{align*}
Therefore, $\Delta$ has a counit $\varepsilon_\Delta$, defined as the linear map from $\Alg[\emptyset]$ to $\K$ sending $1$ to $1$.

Let us assume that $X=I\sqcup J\sqcup K$. Let $\pi \in \ncp[X]$. We use the notations of Definition \ref{defi2.3}.
\begin{align*}
(\id \otimes \Delta_{J,K})\circ \Delta_{I,J\sqcup K}(\pi)&=\begin{cases}
\displaystyle \pi_{\mid I}\otimes \prod_{i\in [p]}^\cdot (\pi_{\mid (J\sqcup K)_i})_{\mid J}\otimes \prod_{i\in [p]}^\cdot (\pi_{\mid_ (J\sqcup K)_i})_{\mid_\cdot K}\\
\hspace{.5cm}\mbox{if $J\sqcup K$ ideal of $\pi$ and $K\cap (J\sqcup K)_p$ ideal of $\pi_{\mid (J\sqcup K)_p}$ for any $p$},\\
0\mbox{ otherwise}
\end{cases}\\
&=\begin{cases}
\pi_{\mid I}\otimes \pi_{\mid_\cdot J} \otimes \pi_{\mid_\cdot K} \mbox{ if $J\sqcup K$ and $K$ ideals of $\pi$},\\
0\mbox{ otherwise}
\end{cases}\\
&=\begin{cases}
(\pi_{\mid I\sqcup J})_{\mid I}\otimes (\pi_{\mid I\sqcup J})_{\mid_\cdot J} \otimes \pi_{\mid_\cdot K}\mbox{ if $J\sqcup K$ and $K$ ideals of $\pi$},\\
0\mbox{ otherwise}
\end{cases}\\
&=(\Delta_{I,J}\otimes \id)\circ \Delta_{I\sqcup J,K}(\pi).
\end{align*}
By multiplicativity, $\Delta$ is coassociative. \end{proof}

Let us now apply the bosonic Fock functor $\calF$  on $\Alg$, see \cite{Aguiar2010,Foissy41}. As an algebra, $\calF[\Alg]$ is described as follows.

\begin{defi}
We denote by $\alg$ be free commutative algebra generated by $\ncp$. Its product will be denoted by $\cdot$.
A basis of $\alg$ is given by commutative monomials in $\ncp$, and the unit is the empty monomial $1$.
\end{defi}

Observe that in $\alg$,
\begin{align*}
\ncun\cdot \ncdeuxun&=\ncdeuxun\cdot \ncun,&\ncun\cdot \ncdeuxun&\neq \nctroistrois,&\ncdeuxun \cdot \ncun&\neq \nctroisdeux,&\nctroistrois&\neq \nctroisdeux.
\end{align*}
Applying the functor $\calF$  on $\Alg$, we obtain the following coproduct  on the algebra $\alg$: for any noncrossing partition $\pi$ of $[n]$,
\begin{align}
\label{Eq1}\Delta(\pi)&=\sum_{\mbox{\scriptsize $J$ ideal of $\pi$}}\pi_{\mid [n]\setminus J}\otimes \pi_{\mid_\cdot J}.
\end{align}
Colored version can also obtained, using the colored Fock functor of \cite{Foissy41} -- we won't use them in this paper.

\begin{example}\begin{align*}
\Delta(\ncun)&=\ncun\otimes 1+1\otimes \ncun,\\
\Delta(\ncdeuxun)&=\ncdeuxun \otimes 1+1\otimes \ncdeuxun,\\
\Delta(\ncdeuxdeux)&=\ncdeuxdeux\otimes 1+1\otimes \ncdeuxdeux+2\ncun\otimes \ncun,\\
\Delta(\nctroisun)&=\nctroisun\otimes 1+1\otimes \nctroisun,\\
\Delta(\nctroisdeux)&=\nctroisdeux\otimes 1+1\otimes \nctroisdeux+\ncdeuxun\otimes \ncun+\ncun\otimes \ncdeuxun,\\
\Delta(\nctroistrois)&=\nctroistrois\otimes 1+1\otimes \nctroistrois+\ncdeuxun\otimes \ncun+\ncun\otimes \ncdeuxun,\\
\Delta(\nctroisquatre)&=\nctroisquatre\otimes 1+1\otimes \nctroisquatre+\ncdeuxun\otimes \ncun,\\
\Delta(\nctroiscinq)&=\nctroiscinq\otimes 1+1\otimes \nctroiscinq+\ncun\otimes (\ncun\cdot \ncun+2\ncdeuxdeux)+3\ncdeuxdeux\otimes \ncun.
\end{align*}
This coproduct is also introduced in \cite{Foissy38}, with operadic methods. 
\end{example}

\subsection{A link with the Faà di Bruno Hopf algebra}

The Faà di Bruno group is the group of formal series 
\[\GFDB=\left\{x+\sum_{n=1}^\infty a_nx^{n+1}\mid (a_n)_{n\geq 1}\in \K^{\N^*}\right\}\]
with the composition of formal series. For any $n\geq 1$, we consider the map
\[x_n:\left\{\begin{array}{rcl}
\GFDB&\longrightarrow&\K\\
\displaystyle x+\sum_{n=1}^\infty a_nx^{n+1}&\longmapsto&a_n.
\end{array}\right.\]
The Faà di Bruno Hopf algebra is the algebra $\HFDB=\K[x_n\mid n\geq 1]$ generated by these maps, with the coproduct defined by
\begin{align*}
&\forall n\geq 1,\: \forall f,g\in \GFDB,&\Delta(x_n)(f,g)&=x_n(f\circ g).
\end{align*}
For example, if
\begin{align*}
f&=x+\sum_{n=1}^\infty a_nx^{n+1},&g&=x+\sum_{n=1}^\infty b_nx^{n+1},
\end{align*}
then 
\begin{align*}
f\circ g&=x+(a_1+b_1)x^2+(a_2+b_2+2a_1b_1)x^3+(a_3+b_3+2a_1b_2+3a_2b_1+a_1b_1^2)x^4+\cdots,
\end{align*}
which gives
\begin{align*}
\Delta(x_1)&=x_1\otimes 1+1\otimes x_1,\\
\Delta(x_2)&=x_2\otimes 1+1\otimes x_2+2x_1\otimes x_1,\\
\Delta(x_3)&=x_3\otimes 1+1\otimes x_3+2x_1\otimes x_2+3x_2\otimes x_1+x_1\otimes x_1^2.
\end{align*}
It is a graded Hopf algebra, with $x_n$ homogeneous of degree $n$ for any $n\geq 1$. 

\begin{notation}\label{not2.4}
For any $n\geq 1$, we denote by $J_n$ the noncrossing partition $\{\{1\},\ldots,\{n\}\}$:
\begin{align*}
J_1&=\ncun,&J_2&=\ncdeuxdeux,&J_3&=\nctroiscinq,&J_4&=\ncquatrequatorze,\ldots
\end{align*}\end{notation}

\begin{prop}\label{prop2.8}
The following map defines a Hopf algebra morphism:
\[\phiFDB:\left\{\begin{array}{rcl}
(\HFDB,m,\Delta)&\longrightarrow&(\alg,\cdot,\Delta)\\
x_n&\longmapsto&J_n.
\end{array}\right.\]
\end{prop}

\begin{proof}
Let us first give formulas for the coproduct of $x_n$ in $\HFDB$. If $(a_n)_{n\geq 1}$ and $(b_n)_{n\in \geq 1}$ are two sequence, putting $a_0=b_0=1$ in order to simplify,
\begin{align*}
\left(\sum_{k=1}^\infty a_{k-1}x^k\right)\circ \left(\sum_{l=1}^\infty b_{l-1}x^l\right)&=\sum_{k=1}^\infty a_{k-1}\left(\sum_{l=1}^\infty b_{l-1}x^l\right)^k\\
&=\sum_{n=0}^\infty \sum_{1k_0+\cdots+(n+1)k_n=n+1}\frac{(k_0+\cdots+k_n)!}{k_0!\ldots k_n!}a_{k_0+\cdots+k_n-1}b_1^{k_1}\ldots b_n^{k_n}x^{n+1}. 
\end{align*}
We deduce that for any $n\geq 1$,
\begin{align}
\nonumber \Delta(x_n)&=\sum_{1k_0+\cdots+(n+1)k_n=n+1}\frac{(k_0+\cdots+k_n)!}{k_0!\ldots k_n!}x_{k_0+\cdots+k_n-1}\otimes x_1^{k_1}\ldots x_n^{k_n}\\
\label{Eq2}&=\sum_{1k_0+\cdots+(n+1)k_n=n+1}\frac{(k_0+\cdots+k_n)!}{k_0!\ldots k_n!}x_{n-1k_1-\cdots-nk_n}\otimes x_1^{k_1}\ldots x_n^{k_n},
\end{align}
with the convention $x_0=1$ when $(k_0,\ldots,k_n)=(0,\ldots,0,1)$. \\

Let us now give a formula for the coproduct of $J_n$. There exists scalars $\lambda_{k_0,\ldots,k_n}$ such that for any $n\geq 1$,
\begin{align*}
\Delta(J_n)&=\sum_{1k_0+\cdots+(n+1)k_n=n+1}\lambda_{k_0,\ldots,k_n}J_{n-1k_1-\cdots-nk_n}\otimes J_1^{k_1}\ldots J_n^{k_n}.
\end{align*}
Let us fix $(k_0,\ldots,k_n)$ such that $1k_0+\cdots+(n+1)k_n=n+1$.  Then $\lambda_{k_0,\ldots,k_n}$ is the number of way of inserting $k_1$ copies of $J_1$, $\ldots$, $n+1$ copies of $J_n$ between the blocks of $J_{n-1k_1-\cdots-nk_n}$,
such that any two of these copies are separated by at least one block of $J_{n-1k_1-\cdots-nk_n}$. Therefore,
\begin{align*}
\lambda_{k_0,\ldots,k_n}&=\binom{n-1k_1-\cdots-nk_n+1}{k_1+\cdots+k_n} \frac{(k_1+\cdots+k_n)!}{k_1!\ldots k_n!}\\
&=\binom{k_0+\cdots+k_n}{k_1+\cdots+k_n} \frac{(k_1+\cdots+k_n)!}{k_1!\ldots k_n!}\\
&=\frac{(k_0+\cdots+k_n)!}{k_0!(k_1+\cdots+k_n)!} \frac{(k_1+\cdots+k_n)!}{k_1!\ldots k_n!}\\
&=\frac{(k_0+\cdots+k_n)!}{k_0!\ldots k_n!}.
\end{align*}
The first binomial coefficient corresponds to the choices of the $k_1+\cdots+k_n$ spaces in the $n-1k_1-\cdots-nk_n+1$ available in $J_{n-1k_1-\cdots-nk_n}$  where we shall insert the blocks,
and the second factor takes in account the possible orders on the copies of $J_i$ we insert. We finally obtain
\begin{align}
 \label{Eq3}\Delta(J_n)&=\sum_{1k_0+\cdots+(n+1)k_n=n+1}\frac{(k_0+\cdots+k_n)!}{k_0!\ldots k_n!}J_{n-1k_1-\cdots-nk_n}\otimes J_1^{k_1}\ldots J_n^{k_n},
\end{align}
with the convention $J_0=1$ when $(k_0,\ldots,k_n)=(0,\ldots,0,1)$. Comparing (\ref{Eq2}) and (\ref{Eq3}), we immediately obtain that $\phiFDB$ is a Hopf algebra morphism. 
\end{proof}

\subsection{The extraction-contraction coproduct}

Recall that $\eq$ is the species which associates to any finite set $X$ the set of equivalence relations on $X$.

\begin{defi}
Let $\pi$ be an $X$-indexed noncrossing partition and $\sim \in \eq[X]$.
\begin{enumerate}
\item For any $\overline{x}\in X/\sim$, we put $\displaystyle \pi_{\overline{x}}=\bigsqcup_{y\in \overline{x}}\pi_y$, and $\pi/\sim=\left(\pi_{\overline{x}}\right)_{\overline{x}\in X/\sim}$. This defines an $X/\sim$-indexed partition.
\item We shall say that $\sim \in \eq_c[\pi]$ if $X/\sim \in \ncp[X/\sim]$. 
\item We put $\displaystyle \pi\mid_\cdot \sim=\prod_{\overline{x}\in X/\sim}^\cdot \pi_{\mid \overline{x}}$. This defines an element of $\Alg[X]$. 
\end{enumerate}\end{defi}

\begin{prop}\label{prop2.10}
For any $X$-indexed noncrossing partition, for any $\sim\in \eq[X]$, we put
\[\delta_\sim(\pi)=\begin{cases}
\pi/\sim\otimes \pi\mid_\cdot \sim \mbox{ if }\sim\in \eq_c[\pi],\\
0\mbox{ otherwise}.
\end{cases}\]
We extend this by multiplicativity to $\Alg$. More precisely, if $X=X_1\sqcup \ldots \sqcup X_k$ and if for any $i$, $\pi_i\in \ncp[X_i]$, then, for any $\sim\in \eq[X]$,
\[\delta_\sim(\pi_1\cdot \ldots \cdot \pi_k)=\begin{cases}
\delta_{\sim \cap X_1^2}(\pi_1)\cdot \ldots \cdot \delta_{\sim \cap X_1^2}(\pi_k) \mbox{ if }\sim=(\sim \cap X_1^2)\sqcup\ldots \sqcup (\sim \cap X_k^2),\\
0\mbox{ otherwise}.
\end{cases}\]
Then $\delta$ is a contraction-extraction coproduct on $\Alg$ in the sense of \cite{Foissy41}, compatible with $m$ and $\Delta$. 
\end{prop}

\begin{proof}
Let $\pi\in \ncp[X]$ and let $\sim,\sim'\in \eq[X]$, with $\sim'\subseteq \sim$. We denote  by $\overline{\sim}$ the equivalence on $X/\sim'$ induced by $\sim$. 
\begin{align}
\label{Eq4}(\delta_\sim\otimes \id)\circ \delta_{\sim'}(\pi)&=\begin{cases}
(\pi/\sim')/\overline{\sim}\otimes(\pi/ \sim')|\overline{\sim}\otimes \pi\mid_\cdot \sim'\\
\hspace{.5cm}\mbox{if $\sim'\in \eq_c[\pi]$ and $\overline{\sim}\in \eq_c[\pi/\sim']$},\\
0\mbox{ otherwise}.
\end{cases}\\
\label{Eq5} (\id \otimes \delta_{\sim'})\circ \delta_\sim(\pi)&=\begin{cases}
\pi/\sim\otimes (\pi\mid_\cdot \sim)/\sim'\otimes (\pi\mid_\cdot \sim)\mid_\cdot \sim'\\
\hspace{.5cm}\mbox{if $\sim\in \eq_c[\pi]$ and $\sim'\in \eq_c[\pi\mid_\cdot \sim]$},\\
0\mbox{ otherwise}.
\end{cases}\end{align}
Note that, as indexed partitions or monomials of indexed partitions,
\begin{align*}
(\pi/\sim')/\overline{\sim}&=\pi/\sim,&(\pi/\sim')\mid \overline{\sim}&=(\pi\mid\sim)/\sim',&(\pi\mid_\cdot \sim)\mid_\cdot \sim'&=\pi\mid_\cdot \sim'.
\end{align*}
Let us assume that $\sim'\in \eq_c[\pi]$ and $\overline{\sim}\in \eq_c[\pi/\sim']$. Then $(\pi/\sim')/\overline{\sim}=\pi/\sim$ is noncrossing, so  $\sim\in \eq_c[\pi]$.
Moreover, $(\pi/\sim')\mid \overline{\sim}=(\pi\mid\sim)/\sim'$ is noncrossing as $\pi/\sim'$ is noncrossing, so $\sim'\in \eq_c[\pi\mid_\cdot \sim]$.
Conversely, let us assume that $\sim\in \eq_c[\pi]$ and $\sim'\in \eq_c[\pi\mid_\cdot \sim]$. Then $(\pi/\sim')/\overline{\sim}=\pi/\sim$ is noncrossing, so  $\overline{\sim}\in \eq_c[\pi/\sim']$.
Moreover, $(\pi/\sim')\mid \overline{\sim}=(\pi\mid\sim)/\sim'$ is noncrossing, so $\sim'\in \eq_c[\pi \mid_\cdot \sim]$. As $\sim'\subseteq \sim$, $\sim' \in \eq_c[\pi]$. 

Finally, the two conditions in (\ref{Eq4}) and (\ref{Eq5}) are equivalent. This gives the coassociativity of $\delta$ when applied to any noncrossing partition, then, by multiplicativity, the coassociativity of $\delta$.\\

Let now  $\pi\in \ncp[X]$ and let $\sim,\sim'\in \eq[X]$, without $\sim'\subseteq \sim$. This means that at least one class of $\sim$ is not the union of the classes of $\sim'$ that it contains.
Hence, denoting by $\overline{x}_1,\ldots,\overline{x}_k$ the classes of $\sim$, $\sim'$ is not equal to the union of the parts $\sim'\cap \overline{x}_i^2$. Consequently,
\[\delta_{\sim'}(\pi\mid_\cdot \sim)=\delta_{\sim'}(\pi_{\mid \overline{x}_1}\cdot \ldots \cdot \pi_{\mid \overline{x}_k})=0.\]
We obtain that $(\id \otimes \delta_{\sim'})\circ \delta_\sim=0$. \\

Let $\varepsilon_\delta$ be the twisted algebra morphism from $\Alg$ to $\com$ defined by
\begin{align*}
&\forall \pi\in \ncp[X],&\varepsilon_\delta[X](\pi)&=\begin{cases}
1\mbox{ if $\pi$ has only one block},\\
0\mbox{ otherwise}.
\end{cases}\end{align*}
Let $\pi\in \ncp[X]$. The unique $\sim\in \eq_c[X]$ such that $\pi\mid_\cdot \sim$ is a monomial which all factors are noncrossing partitions with only one block is the equality of $X$, which we denote by $\sim_1$. Moreover, $\pi/\sim_1=\pi$.
Hence, $\epsilon_\delta$ is a right counit for $\delta$. The unique $\sim\in \eq_c[X]$ such that $\pi/\sim$ has only one block is the equivalence $\sim_0$ on $X$ with only one class. 
As $\pi\mid_\cdot \sim_0=\pi$, $\varepsilon_\delta$ is a left counit for $\delta$. \\

Let $X$ and $Y$ be two disjoint sets, $\pi\in \ncp[X\sqcup Y]$, $\sim_X\in \eq[X]$ and $\sim_Y \in \eq[Y]$. We put $\sim=\sim_X\sqcup \sim_Y$. 
\begin{align}
\label{Eq6} (\Delta_{X/\sim_X,Y/\sim_Y}\otimes \id)\circ \delta_\sim(\pi)&=\begin{cases}
(\pi/\sim)_{\mid X}/\sim_X\otimes (\pi/\sim)_{\mid_\cdot Y}\otimes \pi\mid_\cdot \sim\\
\hspace{.5cm}\mbox{if $\sim\in \eq_c[\pi]$ and $Y/\sim_Y$ is an ideal of $\pi/\sim$},\\
0\mbox{ otherwise}.
\end{cases}\\
\label{Eq7} m_{1,3,24}\circ (\delta_X\otimes \delta_Y)\circ \Delta_{X,Y}(\pi)&=\begin{cases}
(\pi_{\mid X})/\sim_X \otimes (\pi_{\mid_\cdot Y})/\sim_Y\otimes (\pi_{\mid X})|\sim_X \cdot (\pi_{\mid_\cdot Y})|\sim_Y\\
\hspace{.5cm}\mbox{if $Y$ is an ideal of $\pi$},\\
\hspace{.5cm} \mbox{$\sim_X\in \eq_c[\pi_{\mid X}]$ and $\sim_Y\in \eq_c[\pi_{\mid_\cdot Y}]$},\\
0\mbox{ otherwise}.
\end{cases}\end{align}

Note that, as indexed partitions or as monomials of indexed partitions,
\begin{align*}
(\pi/\sim)\mid (X/\sim_X)&=(\pi_{\mid X})/\sim_X,&(\pi/\sim)_{\mid_\cdot Y}&=(\pi_{\mid_\cdot Y})/\sim_Y,&\pi\mid_\cdot \sim&=((\pi_{\mid X})|\sim_X) \cdot ((\pi_{\mid_\cdot Y})|\sim_Y).
\end{align*}

Let us assume that  $\sim\in \eq_c[\pi]$ and $Y/\sim_Y$ is an ideal of $\pi/\sim$. Let $\overline{x},\overline{x}' \in \pi$, such that $\overline{x}\leq_\pi \overline{x}'$ and $\overline{x}\in Y$. We denote by $\varpi:(X\sqcup Y)\longrightarrow (X\sqcup Y)/\sim$
the canonical surjection. Then $\varpi(\overline{x})\leq_{\pi/\sim}\varpi(\overline{x}')$ in $\pi/\sim$. As $Y/\sim_Y$ is an ideal of $\pi/\sim$, $\varpi(\overline{x}')\in Y/\sim_Y$, so $\overline{x}'\in Y$. 
By restriction, $(\pi\mid X)/\sim_X=(\pi/\sim)\mid X/\sim_X$ is noncrossing as $\pi/\sim$ is noncrossing, so $\sim_X\in \eq_c[\pi_{\mid X}]$. Similarly, $\sim_Y\in \eq_c[\pi_{\mid_\cdot Y}]$.

Let us now assume that $Y$ is an ideal of $\pi$, $\sim_X\in \eq_c[\pi_{\mid X}]$ and $\sim_Y\in \eq_c[\pi_{\mid_\cdot Y}]$. Let $\overline{x},\overline{x}'\in (X\sqcup Y)/\sim$. 
\begin{itemize}
\item If $\overline{x}\in Y/\sim_Y$ and $\overline{x}'\in X/\sim_X$, as $\sim_Y \in \eq_c[\pi_{\mid_\cdot Y}]$, $\sim_Y$ is compatible with the connected components of $Y$, as defined in Definition \ref{defi2.3}. 
Hence, $\pi_{\overline{x}'}$ cannot cross $\pi_{\overline{x}}$. Similarly, if $\overline{x}\in X/\sim_X$ and $\overline{x}'\in Y/\sim_Y$, then $\pi_{\overline{x}}$ cannot cross $\pi_{\overline{x}'}$.
\item If both $\overline{x},\overline{x}'$ belong to $X/\sim_X$, As $(\pi_{\mid X})/\sim_X=(\pi/\sim)_{\mid X/\sim_X}$ is noncrossing, $\pi_{\overline{x}}$ cannot cross $\pi_{\overline{x}'}$. 
Similarly, if  both $\overline{x},\overline{x}'$ belong to $Y/\sim_Y$, then $\pi_{\overline{x}}$ cannot cross $\pi_{\overline{x}'}$. 
\end{itemize}
So $\sim\in \eq_c[\pi]$. By contraction, $Y/\sim_Y$ is an ideal of $\pi/\sim$. \\

Finally, the two conditions in (\ref{Eq6}) and (\ref{Eq7}) are equivalent. This gives the compatibility of $\delta$ and $\Delta$. \end{proof}

Applying the functor $\calF$ on $\Alg$, we obtain the following coproduct on $\alg$, see \cite{Foissy41}. For any noncrossing partition $\pi$,
\begin{align}
\label{Eq8}\delta(\pi)&=\sum_{\sim \in \eq_c[\pi]}\pi/\sim\otimes \pi\mid_\cdot \sim.
\end{align}
Colored versions could also be obtained with colored Fock functors \cite{Foissy41} (when the space of colors is a commutative and cocommutative, not necessarily unitary bialgebra). 

\begin{example}\begin{enumerate}
\item n $\alg$,
\begin{align*}
\delta(\ncun)&=\ncun\otimes \ncun,\\
\delta(\ncdeuxun)&=\ncdeuxun\otimes \ncdeuxun,\\
\delta(\ncdeuxdeux)&=\ncdeuxdeux\otimes \ncun\cdot \ncun+\ncdeuxun\otimes \ncdeuxdeux,\\
\delta(\nctroisun)&=\nctroisun\otimes \nctroisun,\\
\delta(\nctroisdeux)&=\nctroisdeux\otimes \ncdeuxun\cdot\ncun+\nctroisun\otimes \nctroisdeux,\\
\delta(\nctroistrois)&=\nctroistrois\otimes \ncdeuxun\cdot\ncun+\nctroisun\otimes \nctroistrois,\\
\delta(\nctroisquatre)&=\nctroisquatre\otimes \ncdeuxun\cdot\ncun+\nctroisun\otimes \nctroisquatre,\\
\delta(\nctroiscinq)&=\nctroiscinq\otimes \ncun \cdot \ncun\cdot\ncun+(\nctroisdeux+\nctroistrois+\nctroisquatre)\otimes \ncun\cdot \ncdeuxdeux+\nctroisun\otimes \nctroiscinq.
\end{align*}
This is the second coproduct of \cite{Foissy38}. 
\item Recall that $J_n$ is defined in Notation \ref{not2.4}. For any $n\geq 1$,
\[\delta(J_n)=\sum_{\mbox{\scriptsize $\pi$ noncrossing partition of $[n]$}} \pi\otimes \prod^\cdot_{b\in \pi}J_{|b|}.\]
\end{enumerate}\end{example}

\section{Fundamental polynomial invariant}

\subsection{Construction and combinatorial interpretation}

As a double bialgebra, $\alg$ is the symmetric algebra generated by noncrossing partitions, with the two multiplicative coproducts defined by (\ref{Eq1}) and (\ref{Eq8}). 
We denote by $\phi_\ncp$ the unique double bialgebra morphism from $\alg$ to $\K[X]$, see Theorem \ref{theo1.4}. For any $\pi\in \ncp$,
\begin{align*}
\phi_\ncp(\pi)&=\sum_{k=1}^\infty \sum_{\substack{\pi=I_1\sqcup \ldots \sqcup I_k,\\
\forall p\in [k], \: I_p\sqcup\ldots \sqcup I_k\:\mbox{\scriptsize ideal of }\pi}} \epsilon_\delta(\pi_{\mid I_1})\epsilon_\delta(\pi_{\mid_\cdot I_2})\ldots \epsilon_\delta(\pi_{\mid_\cdot I_k}) H_k(X)\\
 &=\sum_f \epsilon_\delta(\pi_{\mid f^{-1}(1)})\epsilon_\delta(\pi_{\mid_\cdot f^{-1}(2)})\ldots \epsilon_\delta(\pi_{\mid_\cdot f^{-1}(\max(f))}) H_{\max(f)}(X),
\end{align*}
where the sum runs over all surjective maps $f:\pi\longrightarrow [\max(f)]$ such that for any $b,b'\in \pi$,
\[b\leqslant_\pi b'\Longrightarrow f(b)\leqslant f(b').\]
Moreover, due to the definition of $\epsilon_\delta$, the term corresponding to $f$ is not zero if, and only if, for any $i>1$, $\pi_{\mid_\cdot f^{-1}(i)}$ is a monomial of noncrossing partitions with only one block each,
and $\pi_{\mid f^{-1}(1)}$ is a noncrossing partition with only one block. If this holds, then the contribution of $f$ is $1$. Reformulating:

\begin{prop}\label{prop3.1}
For any noncrossing partition $\pi$, for any $n\in \N_{>0}$, $\phi_\ncp(G)(n)$ is the number of maps $f:\pi\longrightarrow [n]$ such that, for any $b,b'\in \pi$,
\begin{itemize}
\item If $b\leqslant_\pi b'$ and $b\neq b'$, then $f(b)<f(b')$.
\item If $f(b)=f(b')$ and $\max(b)<\min(b')$, then there exists $b''\in \pi$ such that 
\begin{align*}
&]\max(b),\min(b')[\cap b''\neq \emptyset&&\mbox{and}&f(b'')<f(b)&=f(b'). 
\end{align*}
\end{itemize}
Such a map $f$ will be called a valid $n$-coloration of $\pi$.
\end{prop}

\begin{example}\begin{enumerate}
\item A valid $n$-coloration of $\pi=\ncquatretreize$ is a map $f:[3]\longrightarrow [n]$ such that 
\begin{align*}
f(1)&<f(2),& f(1)&<f(3),&f(2)&\neq f(3). 
\end{align*}
Therefore, for any $n\geq 1$,
\[\phi_\ncp(\ncquatretreize)(n)=\frac{n(n-1)(n-2)}{3}.\]
We obtain that
\[\phi_\ncp(\ncquatretreize)=\frac{X(X-1)(X-2)}{3}.\]
\item A valid $n$-coloration of $\nctroiscinq$ is a map $f:[3]\longrightarrow [n]$ such that
\begin{align*}
f(1)&\neq f(2),&f(2)&\neq f(3),&f(1)=f(3)&\Longrightarrow f(2)<f(1).
\end{align*} 
Therefore, for any $n\geq 1$,
\[\phi_\ncp(\nctroiscinq)(n)=n(n-1)(n-2)+\frac{n(n-1)}{2}=n(n-1)\left(n-\frac{3}{2}\right).\]
The first term correspond to the valid $n$-colorations with $f(1)\neq f(3)$ and the second with $f(1)=f(3)$. We obtain that 
\[\phi_\ncp(\nctroiscinq)=X(X-1)\left(X-\frac{3}{2}\right).\]
\item See the Table at the end of the article for more examples of $\phi_\ncp(\pi)$.
\end{enumerate}\end{example}

Let us now give an inductive way to compute these chromatic polynomials.

\begin{notation}
Let $\pi$ be a noncrossing partition. We denote by $\base(\pi)$ the set of blocks of $\pi$ which are minimal for the partial order $\leq_\pi$.
\end{notation}

\begin{prop} \label{prop3.2}
Let $\pi$ be a noncrossing partition. Then
\begin{align*}
\phi_\ncp(\pi)(X+1)&=\phi_\ncp(\pi)\\
&+\sum_{b\in \base(\pi)}
\phi_\ncp\left( \pi_{\mid [1,\min(b)[}\right)\phi_\ncp\left( \pi_{\mid_\cdot [\min(b),\max(b)]\setminus b}\right) \phi_\ncp\left(\pi_{\mid]\max(b),\infty[}\right).
\end{align*}
Together with $\phi_\ncp(\pi)(0)=0$, this characterizes $\phi_\ncp(\pi)$. 
\end{prop}

\begin{proof}
Let $n\geq 1$ and let $f:\pi\longrightarrow [n+1]$ be a $(n+1)$-valid coloration of $\pi$. By the increasing condition, $f^{-1}(1)$ is included in $\base(\pi)$.
If it is not empty, by the second condition, it is reduced to a singleton. In the first case, it corresponds, up to a lift of 1, to a $n$-valid coloration of $\pi$.
In the second case, denoting $\{b\}=f^{-1}(1)$, it corresponds to a $n$-valid coloration of $ \pi_{\mid [1,\min(b)[}\cdot \pi_{\mid [\min(b),\max(b)]\setminus b}\cdot\pi_{\mid]\max(b),\infty[}$. 
Summing, we obtain the announced result. 
\end{proof}

\subsection{Antipode}

In order to compute the antipode, we put, for any noncrossing partition $\pi$,
\[\mu_\ncp(\pi)=\phi_\ncp(\pi)(-1).\]
This defines a character of $\alg$. Theorem \ref{theo1.3} states that the antipode of $(\alg,m,\Delta)$ is
\[S=(\mu_\ncp\otimes \id_{\alg})\circ \delta.\]

\begin{notation} 
We denote by $\cat_n=\displaystyle \dfrac{1}{n+1}\binom{2n}{n}$ is the $n$-th Catalan number.
\[\begin{array}{|c||c|c|c|c|c|c|c|c|c|c|c|}
\hline n&0&1&2&3&4&5&6&7&8&9&10\\
\hline\hline\cat_n&1&1&2&5&14&42&132& 429&1430& 4862&16796\\
\hline \end{array}\]
Recall that
\begin{align}
\label{Eq9}&\forall k\geq 1,&\cat_k&=\sum_{i=1}^{k}\cat_{i-1}\cat_{k-i}.
\end{align}
The number of noncrossing partitions of $[n]$ is $\cat_n$. The formal series of Catalan numbers is
\[C=\sum_{n=0}^\infty \cat_nX^n=\dfrac{1-\sqrt{1-4X}}{2X}.\]
For more details, see Entry A000108 of the OEIS \cite{Sloane}.  
\end{notation}

\begin{lemma}\label{lemma3.3}
The character $\mu_\ncp$ on $\alg$ can be computed by induction on the number of blocks, putting for any noncrossing partition $\pi$ of $[n]$,
\[\mu_\ncp(\pi)=(-1)^{|\base(\pi)|}\cat_{|\base(\pi)|} \mu_\ncp\left(\pi_{\mid_\cdot [n] \setminus \base(\pi)}\right).\]
The antipode $S$ of $(\alg,m,\Delta)$ is given on any noncrossing partition $\pi$ by
\[S(\pi)=\sum_{\sim \in \eq_c[\pi]}\mu_\ncp(\pi/\sim)\pi\mid_\cdot \sim.\]
\end{lemma}

\begin{proof}
We evaluate the formula of Proposition \ref{prop3.2} in $X=-1$. We obtain
\[\mu_\ncp(\pi)=-\sum_{b\in \base(\pi)}
\mu_\ncp( \pi_{\mid [1,\min(b)[})\mu_\ncp( \pi_{\mid_\cdot [\min(b),\max(b)]\setminus b}) \mu_\ncp(\pi_{\mid]\max(b),\infty[}).\]
We prove the formula for $\mu_\ncp(\pi)$ by induction on $k=|\base(\pi)|$. We put $\base(\pi)=\{b_1,\ldots,b_k\}$
with $\max(b_i)+1=\min(b_{i+1})$ for any $i\in [k-1]$. If $k=1$, then $\mu_\ncp(\pi_{\mid \base(pi)})=-1$,
which immediately gives the result. If $k\geq 2$, then, by the induction hypothesis,
\begin{align*}
\mu_\ncp(\pi)&=-\sum_{i=1}^k \mu_\ncp(\pi_{\mid [1,\max(b_{i-1})]})\mu_\ncp(\pi_{\mid_\cdot [\min(b_i),\max(b_i)])\setminus b_i})\mu_\ncp(\pi_{\mid [\min(b_{i+1},+\infty[})\\
&=-\sum_{i=1}^k (-1)^{i-1}\cat_{i-1} \mu_\ncp(\pi_{\mid_\cdot [1,\max(b_{i-1})]\setminus (b_1\cup \ldots \cup b_{i-1})})
\mu_\ncp(\pi_{\mid_\cdot [\min(b_i),\max(b_i)])\setminus b_i})\\
&\hspace{1cm} (-1)^{k-i}\cat_{k-i} \mu_\ncp(\pi_{\mid_\cdot [\min(b_{i+1},+\infty[\setminus (b_{i+1}\cup \ldots \cup b_k)}) \\
&=(-1)^k\mu_\ncp\left(\pi_{\mid_\cdot [n] \setminus \base(\pi)}\right)\sum_{i=1}^k\cat_{i-1}\cat_{k-i}\\
&=(-1)^k\cat_k \mu_\ncp\left(\pi_{\mid_\cdot [n] \setminus \base(\pi)}\right). 
\end{align*}
We used (\ref{Eq9}) for the last equality.
\end{proof}

Let us now solve this induction. We shall need the following combinatorial notions:

\begin{defi}\label{defi3.4}\begin{enumerate}
\item For any $n\geq 1$, let us denote by $\parti(n)$ the set of partitions of $n$:
\[\parti(n)=\{(k_1,\ldots,k_n)\in \N^n\mid 1k_1+\cdots+nk_n=n\}.\]
\item Let $\pi$ a noncrossing partition on $[n]$. 
\begin{enumerate}
\item For any $i\in [n]$, we denote by $p_i(\pi)$ the number of blocks of $\pi$ of cardinality $i$ and put
\[p(\pi)=(p_1(\pi),\ldots,p_n(\pi))\in \parti(n).\]
\item We define an equivalence $\sim_\pi $on $\pi$ be the following:
for any $b,b'\in \pi$, $b\sim_\pi b'$ if there exists a sequence $(b_1,\ldots,b_k)$ of elements of $\pi$ such that:
\begin{itemize}
\item ($b_1=b$ and $b_k=b'$) or $(b_1=b'$ and $b_k=b$).
\item For any $i\in [k-1]$, $\max(b_i)+1=\min(b_{i+1})$.
\end{itemize}\end{enumerate}\end{enumerate}\end{defi}

\begin{example}
Let us consider 
\[\pi=\ncextrois.\]
We index the blocs of $\pi$ in order to simplify the incoming description of $\sim_\pi$: The classes of $\sim_\pi$ are $\{1,7,8\}$, $\{5,6\}$ and the singletons $\{2\}$, $\{3\}$, $\{4\}$. Moreover, $p(\pi)=(4,2,1,0,1,0,0,0)$. 
\end{example}

\begin{prop}\label{prop3.5}
Let $\pi$ be a noncrossing partition of $[n]$. Then 
\[\mu_\ncp(\pi)=(-1)^{|\pi|}\prod_{X\in \pi/\sim_\pi} \cat_{|X|}.\]
\end{prop}

\begin{proof}
We put $\pi_{\mid_\cdot [n]\setminus \base(\pi)}=\pi_1\cdot \ldots \cdot \pi_k$. Then the elements of $\pi/\sim_\pi$ are $\base(\pi)$ and the elements of $\pi_1/\sim_{\pi_1},\ldots, \pi_k/\sim_{\pi_k}$. The result comes from Lemma \ref{lemma3.3} by a direct induction. 
\end{proof}

\begin{example}
See the Table at the end of the article for examples of $\mu_\ncp(\pi)$.
\end{example}

\begin{example} \label{ex3.2} Here are a few examples of antipodes.
\begin{align*}
S(\ncun)&=-\ncun,\\
S(\ncdeuxdeux)&=-\ncdeuxdeux+2\:\ncun\cdot\ncun,\\
S(\nctroiscinq)&=-\nctroiscinq+(1+2+2)\:\ncun\cdot \ncdeuxdeux-5\:\ncun\cdot \ncun\cdot\ncun\\
&=-\nctroiscinq+5\:\ncun\cdot \ncdeuxdeux-5\:\ncun\cdot \ncun\cdot\ncun,\\
S(\ncquatrequatorze)&=-\ncquatrequatorze+(2+1)\:\ncdeuxdeux \cdot \ncdeuxdeux-(5+5+5+2+2+2)\:\ncun\cdot\ncun\cdot \ncdeuxdeux\\
&+(2+1+1+2)\:\ncun\cdot\nctroiscinq+14\:\ncun\cdot\ncun\cdot\ncun\cdot\ncun\\
&=-\ncquatrequatorze+3\:\ncdeuxdeux \cdot \ncdeuxdeux-21\:\ncun\cdot\ncun\cdot \ncdeuxdeux+6\:\ncun\cdot\nctroiscinq+14\:\ncun\cdot\ncun\cdot\ncun\cdot\ncun.
\end{align*}\end{example}

As a particular case:

\begin{prop}\label{prop3.6}
For any $n\geq 1$,
\begin{align*}
S(J_n)&=\sum_{\mbox{\scriptsize $\pi$ noncrossing partition of $[n]$}} \mu_\ncp(\pi)\prod^\cdot_{b\in \pi}J_{|b|}\\
&=\sum_{(k_1,\ldots,k_n)\in \parti(n)} (-1)^{k_1+\cdots+k_n}\left(
\sum_{\substack{\mbox{\scriptsize $\pi$ noncrossing partition of $[n]$},\\ p(\pi)=(k_1,\ldots,k_n)}}
\prod_{X\in \pi/\sim_\pi} \cat_{|X|}\right)J_1^{k_1}\ldots J_n^{k_n}.
\end{align*}
In particular, the coefficient of $J_1^{\cdot n}$ in $S(J_n)$ is $(-1)^n \cat_n$. 
\end{prop}

Let $(a_n)_{n\geq 1}$ be a sequence of scalars. We can see it as a character of $\HFDB$ thanks to the series
\[f=x+\sum_{n=1}^\infty a_nx^{n+1}\in \GFDB.\]
The coefficients of the inverse for the composition $f^{\circ-1}$ of this formal series are given by the antipode of $\HFDB$. Through the Hopf algebra morphism $\phiFDB$ of Proposition \ref{prop2.8}, it is given by the antipode of $\alg$ applied to the elements $J_n$. 
More precisely, we can write
\[f^{\circ-1}=x+\sum_{n=1}^\infty \left(\sum_{(k_1,\ldots,k)\in \parti(n)} \lambda_{(k_1,\ldots,k_n)}a_1^{k_1}\ldots a_n^{k_n}\right)x^{n+1},\]
where the scalars $\lambda_{(k_1,\ldots,k_n)}$ are given by
\begin{align}
\label{Eq10} \lambda_{(k_1,\ldots,k_n)}&=(-1)^{k_1+\cdots+k_n}
\sum_{\substack{\mbox{\scriptsize $\pi$ noncrossing partition of $[n]$},\\ p(\pi)=(k_1,\ldots,k_n)}}
\prod_{X\in \pi/\sim_\pi} \cat_{|X|}.
\end{align}
These coefficients $\lambda_{(k_1,\ldots,k_n)}$ are given by Entry A111785 of the OEIS \cite{Sloane}, see also  Entry A133437 up to a normalization by factorials. For example,
\begin{align*}
f^{\circ-1}&=x-a_1x^2+(-a_2+a_1^2)x^3+(-a_3+5a_1a_2-5a_1^3)x^4\\
&+(-a_4+3a_2^2-21a_1^2a_2+6a_1a_3+14a_1^4)x^5+\cdots
\end{align*}
In particular, for any $n\geq 1$,
\begin{align*}
\lambda_{(n,0,\ldots,0)}&=(-1)^n \cat_n,& \lambda_{(0,\ldots,0,1)}&=-1.
\end{align*}

\subsection{More results on $\phi_\ncp$ for basic noncrossing partitions}

Let us now give a way to compute the polynomial $\phi_\ncp(\pi)$, when $\pi=\base(\pi)$.

\begin{prop}
Let $\pi$ be a nonempty noncrossing partition such that $\pi=\base(\pi)$.  Then $\phi_\ncp(\pi)$ is a polynomial $P_{|\pi|}$ which depends uniquely on $|\pi|$, characterized by $P_0(X)=1$ and 
\begin{align}
\label{Eq11}
&\forall k\geqslant 1,&&\begin{cases}
P_k(X+1)=P_k(X)+\displaystyle\sum_{i=1}^k P_{i-1}(X)P_{k-i}(X),\\
P_k(0)=0.
\end{cases}\end{align}\end{prop}

\begin{proof} 
This is a direct consequence of Proposition \ref{prop3.2}: if $b$ is the $i$-th block of $\pi$ (totally ordered from left to right), 
then $ \pi_{\mid [1,\min(b)[}$ is a partition $\pi'$ with $i-1$ blocks such that $\base(\pi')=\pi'$,
$\pi_{\mid [\min(b),\max(b)]\setminus b}=1$ and $\pi_{\mid]\max(b),\infty[}$  a partition $\pi''$ with $|\pi|-i$ blocks such that $\base(\pi'')=\pi''$. 
\end{proof}

By their combinatorial nature (and Theorem \ref{theo1.4}), it is natural to decompose these polynomials in the basis of Hilbert polynomials $(H_i(X))_{i\geqslant 0}$. We put
\begin{align}
\label{Eq12} P_n(X)&=\sum_{i=1}^\infty a_{i,n} H_i(X).
\end{align}
Recall that $J_n$ is defined in Notation \ref{not2.4}.
By construction of $P_n(X)=\phi_\ncp(J_n)$, the scalar $a_{i,n}$ is the number of surjective valid  $n$-colorings $c:[n]\longrightarrow [i]$ of $J_n$, 
that is to say the number of surjective maps $c:[n]\longrightarrow[i]$ such that
\begin{align*}
&\forall p,r\in [n],&(p<r \mbox{ and }c(p)=c(r))&\Longrightarrow (\exists q\in ]p,r[, \: c(q)<c(p)).
\end{align*}

The following proposition allows to compute inductively these coefficients $a_{i,n}$:

\begin{prop}\label{prop3.8}
Let $1\leqslant i\leqslant n$. The coefficients $a_{i,n}$ defined by (\ref{Eq12}) can be computed by induction on $i$ by
\begin{align*}
a_{i,n}&=\begin{cases}
\delta_{n,1}\mbox{ if i=1},\\
\displaystyle \sum_{k=1}^{\left[\frac{n+1}{2}\right]}\binom{n-k+1}{k}a_{i-1,n-k}\mbox{ if }i\geqslant 2.
\end{cases}\end{align*}\end{prop}

\begin{proof}
Let $b_{k,n}$ be the number of parts $I$ of $[n]$ with $k$ elements such that for all $i,j\in I$, $|i-j|\neq 1$. This condition will be called the $b$-condition. Then
\begin{align*}
b_{k,l}&=|\{(i_1,\ldots,i_k)\in \N^k, 1\leqslant i_1<i_1+1<i_2<\ldots<i_{k-1}+1<i_k\leqslant n\}|\\
&=|\{(j_1,\ldots,j_k)\in \N^k, 1\leqslant j_1<j_2<\ldots<j_k\leqslant n-k+1\}|\\
&=\binom{n-k+1}{k},
\end{align*}
with $j_p=i_p-p+1$ for any $p$. Note that this coefficient is zero if $k>\left[\dfrac{n+1}{2}\right]$. If $i=1$, then $a_{i,n}=1$ if $n=1$, and $0$ otherwise. 
If $i\geqslant 2$, for any surjective map $c:[n]\longrightarrow [i]$,
\begin{itemize}
\item $c^{-1}(i)$ is a nonempty part $I$ of $[n]$, satisfying the $b$-condition.
\item Identifying $(J_n)_{\mid [n]\setminus I}$ with $J_{n-k}$, where $k=|I|$, the restriction of $c$ to this noncrossing partitions is an element of $PV(J_{n-k})$, of maximum $i-1$.
\end{itemize}
Hence, $\displaystyle a_{i,n}=\sum_{k=1}^n b_{k,n} a_{i-1,n-k}=\sum_{k=1}^{\left[\frac{n+1}{2}\right]}\binom{n-k+1}{k}a_{i-1,n-k}$.
\end{proof}

\begin{prop}
Let $i,n\geq 1$. If $i>n$, then $a_{i,n}=0$. If $n\geq 2^i$, then $a_{i,n}=0$.
\end{prop}

\begin{proof}
The combinatorial interpretation of $a_{i,n}$ (as particular surjective maps from $[n]$ to $[i]$) implies that if $i>n$, $a_{i,n}=0$. Let us consider
\[\calP(X,Y)=\sum_{k=0}^\infty P_k(X)Y^k\in \mathbb{Q}[X][[Y]].\]
Then (\ref{Eq11}) implies that
\[\calP(X+1,Y)=\calP(X,Y)+Y\calP(X,Y)^2.\]
Moreover, $\calP(0,Y)=1$. An easy induction then proves that $\calP(N,Y)$ is a polynomial of $\mathbb{Q}[Y]$ of degree $2^N-1$ for any $N\in \N$. 
Hence, $P_n(N)=0$ if $n\geq 2^N$. This gives, if $n\geq 2^N$,
\[P_n(N)=\sum_{i=1}^\infty a_{i,n}H_i(N)=\sum_{i=0}^N a_{i,n}\binom{N}{i}+0=0.\]
As the coefficients $a_{i,n}$ belong to $\N$, we obtain that $a_{i,n}=0$ if $i\leq N$ and $n\geq 2^N$. In other words, if $n\geq 2^i$, then $a_{i,n}=0$. 
\end{proof}

\begin{example} The following array contains the first values of $a_{i,n}$:
\[\begin{array}{|c||c|c|c|c|c|c|c|c|c|c|}
\hline i\setminus n&1&2&3&4&5&6&7&8&9&10\\
\hline \hline 1&1& 0& 0& 0& 0& 0& 0& 0& 0& 0\\
\hline 2&0& 2& 1& 0& 0& 0& 0& 0& 0& 0\\
\hline 3&0& 0& 6& 10& 8& 4& 1& 0& 0& 0\\
\hline 4&0& 0& 0& 24& 86& 172& 254& 302& 298& 244\\
\hline 5&0& 0& 0& 0& 120& 756& 2734& 7484& 17164& 34612\\
\hline 6&0& 0& 0& 0& 0& 720& 7092& 40148& 172168& 621348\\
\hline 7&0& 0& 0& 0& 0& 0& 5040& 71856& 585108& 3589360\\
\hline 8&0& 0& 0& 0& 0& 0& 0& 40320& 787824& 8720136\\
\hline 9&0& 0& 0& 0& 0& 0& 0& 0& 362880& 9329760\\
\hline 10&0& 0& 0& 0& 0& 0& 0& 0& 0& 3628800\\
\hline \end{array}\]\end{example}

It seems that these coefficients are not so well known. For example, no entry of the OEIS \cite{Sloane} contains the term $172168=a_{6,9}$. 
However, we can obtain explicit formulas for some of these terms, involving multiple nested harmonic sums.

\begin{notation}
if $n,k_1,\ldots,k_p\geq 1$, we put $\displaystyle \zeta_n(k_1,\ldots,k_p)=\sum_{1\leq n_1<\ldots<n_p\leq n} \frac{1}{n_1^{k_1}\ldots n_p^{k_p}}$.
\end{notation}

\begin{cor}\label{cor3.10}
\begin{align*}
&\forall n\geq 1,& a_{n,n}&=n!,\\
&\forall n\geq 2,&a_{n-1,n}&=n!\left(\frac{n+1}{2}-\zeta_n(1)\right),\\
&\forall n\geq 3,&a_{n-2,n}&=n!\left(\zeta_n(1,1)-\frac{n}{2}\zeta_n(1)+\frac{(3n-2)(n^2+n+6)}{24n}\right).
\end{align*}
The sequence $(a_{n-1,n})_{n\geqslant 2}$ is  Entry A31853 in the OEIS \cite{Sloane}.
\end{cor}

\begin{proof} 
For any $n\geqslant 2$, by Proposition \ref{prop3.8},
\[a_{n,n}=\binom{n}{1}a_{n-1,n-1}+0=na_{n-1,n-1},\]
as, if $k\geqslant 2$, $a_{n-1,n-k}=0$. Consequently, $a_{n,n}=n!$. \\

For any $n\geqslant 3$, by Proposition \ref{prop3.8},
\begin{align*}
a_{n-1,n}&=\sum_{k=1}^{\left[\frac{n+1}{2}\right]} \binom{n-k+1}{k} a_{n-2,n-k}\\
&=\binom{n}{1}a_{n-2,n-1}+\binom{n-1}{2}a_{n-2,n-2}+0\\
&=n a_{n-2,n-1}+\frac{(n-1)!(n-2)}{2}.
\end{align*}
Therefore, if $n\geqslant 3$,
\[\frac{a_{n-1,n}}{n!}=\frac{a_{n-2,n-1}}{(n-1)!}+\frac{1}{2}-\frac{1}{n}.\]
As $a_{1,2}=0$, the result follows by  directly.\\

If $n\geqslant 4$, by Proposition \ref{prop3.8},
\begin{align*}
\frac{a_{n-2,n}}{n!}&=\frac{1}{n!}\left(\binom{n}{1}a_{n-3,n-1}+\binom{n-1}{2}a_{n-3,n-2}+\binom{n-2}{3}a_{n-3,n-3}\right)\\
&=\frac{a_{n-3,n-1}}{(n-1)!}+\frac{n-2}{2n}\frac{a_{n-3,n-2}}{(n-2)!}+\frac{(n-3)(n-4)}{6n(n-1)}\frac{a_{n-3,n-3}}{(n-3)!}.
\end{align*}
As $a_{1,3}=0$, and with the preceding formulas, we deduce that
\[\frac{a_{n-2,n}}{n!}=\sum_{k=4}^n \left(\frac{1}{2}-\frac{1}{k}\right)\left(\frac{k-1}{2}-\zeta_{k-2}(1)\right)+\sum_{k=4}^n \frac{(k-3)(k-4)}{6k(k-1)}.\]
The result then follows by tedious manipulations of sums. 
\end{proof}

Let us now give a way to compute the coefficients of $P_n(X)$ in the basis of monomials $(X^k)_{k\in \N}$.

\begin{prop} \label{prop3.11}
Let $M=(M_{k,l})_{k,l\geq 1}$ be the infinite matrix defined by 
\[M_{k,l}=\begin{cases}
\displaystyle \binom{l}{k-l}\mbox{ if }k\geq l,\\
0\mbox{ if }k<l.
\end{cases}\]
Then
\[\left(\begin{array}{c}
P_0(X)\\
P_1(X)\\
P_2(X)\\
\vdots
\end{array}\right)=e^{X\ln(M)}\left(\begin{array}{c}
1\\
0\\
0\\
\vdots
\end{array}\right).\]\end{prop}

\begin{proof}
For any $n\in \N$, we put
\[P_n(X)=\sum_{k=0}^\infty b_{i,n}X^i,\]
and we consider the infinite vectors  $B^{(i)}=(b_{i,n-1})_{n\geq 1}$. Then $B^{(0)}=(\delta_{1,n})_{n\geq 1}$, and for any $n\geq 1$, From Proposition \ref{prop3.8}, if $i\geq 2$, 
\begin{align*}
a_{i,n}&=\sum_{k=1}^{n-1}\binom{n-k+1}{k}a_{i-1,n-k}=\sum_{k=1}^{n-1}\binom{k+1}{n-k}a_{i-1,k}=\sum_{k=0}^{n} \binom{k+1}{n-k}a_{i-1,k}-a_{i-1,n}.
\end{align*}
If $A^{(i)}$ is the infinite vector $(a_{i,n-1})_{n\geq 1}$, we obtain that if $n\geq 2$, then $A^{(i)}=(M-I) A^{(i-1)}$, which is still true if $i=1$. 
Therefore, for any $i\geq 0$, $A^{(i)}=(M-I)^i A^{(0)}$, where $A^{(0)}=(\delta_{1,n})_{n\geq 1}$. \\

Let us now compute the coefficient of $X^i$ in $H_k(X)$ for any $k,n$. We consider the formal series
\[f(X,Y)=\sum_{k\geq 0} H_k(X)Y^k=\sum_{k\geq 0} \binom{X}{k} Y^k=(1+Y)^X \in \mathbb{Q}[[X,Y]].\]
Therefore, 
\begin{align*}
\sum_{k\geq 0} \frac{d^iH_k}{dX^i}(0)Y^k&=\frac{\partial^i f}{\partial X^i}(0,Y)=\left(\ln(1+Y)^i(1+Y)^X\right)_{\mid X=0}=\ln(1+Y)^i.
\end{align*}
Let us put $\displaystyle \ln(1+Y)^i=\sum_k \gamma_{i,k}Y^k$. As for any $n\geq 1$, $\displaystyle P_n=\sum_{k=1}^\infty a_{k,n}H_k(X)$, the coefficient of $X^i$ in $P_n$ is
\begin{align*}
b_{i,n}&=\sum_{k=1}^\infty \frac{1}{i!} \gamma_{i,k} a_{k,n}.
\end{align*}
In other words,
\begin{align*}
B^{(i)}&=\frac{1}{i!}\sum_{k=1}^\infty \gamma_{i,k} A^{(k)}=\frac{1}{i!}\left(\sum_{k=1}^\infty \gamma_{i,k}(M-I)^k\right) B^{(0)}=\frac{1}{i!}\ln(M)^i B^{(0)}. 
\end{align*}
We then obtain
\begin{align*}
\left(\begin{array}{c}
P_0(X)\\
P_1(X)\\
\vdots
\end{array}\right)&=\sum_{i=0}^\infty B^{(i)}X^i=\left(\sum_{i=0}^\infty \frac{X^i}{i!}\ln(M)^iX^i\right) B^{(0)}=e^{X\ln(M)}B^{(0)}. \qedhere
\end{align*}\end{proof}

\begin{remark}\label{rk3.3}
$M$ is the Riordan array of $(1+X,X(1+X))$, see Entries A005119 and A030528 of the OEIS. 
Consequently, the first column of $\ln(M)$, that is to say $B^{(1)}$, gives the coefficients of the infinitesimal generator of $X(1+X)$, up to signs and multiplications by factorials, giving Entry A005119 of the OEIS. 
\end{remark}

\begin{example}\begin{align*}
P_1&=X,\\
P_2&=X^2-X,\\
P_3&=X^3 - \frac{5X^2}{2} + \frac{3X}{2},\\
P_4&=X^4 - \frac{13X^3}{3} + 6X^2 - \frac{8X}{3},\\
P_5&=X^5 - \frac{77X^4}{12} + \frac{89X^3}{6} - \frac{175X^2}{12} +\frac{ 31X}{6},\\
P_6&=X^6 - \frac{87X^5}{10} + \frac{175X^4}{6} -\frac{ 281X^3}{6} + \frac{215X^2}{6} - \frac{157X}{15},\\
P_7&=X^7 - \frac{223X^6}{20} + \frac{1501X^5}{30} - 115X^4 + \frac{851X^3}{6} -\frac{ 1767X^2}{20} +\frac{ 649X}{30}.
\end{align*}\end{example}

\begin{remark}
It could be conjectured seeing these examples that $P_n$ is a polynomial with alternating signs, as are the chromatic polynomials for graphs. This is true till $n=28$ but fails for $n=29$:
\begin{align*}
P_{29}=X^{29} -\frac{6897956948587}{80313433200}X^{28}+\cdots- \frac{14277306976985617719653}{2679791554440}X^2-\frac{37449570182565026}{37182145}X.
\end{align*}\end{remark}

\begin{prop}
Let $k,n\geq 1$. Then $P_k(n)=0$ if, and only if, $k>2^n-1$.
\end{prop}

\begin{proof}
$\Longrightarrow$. Let us assume that $k\leq 2^n-1$ and let us prove that $P_k(n)\neq 0$. 
We define a sequence $D_n=(d_{n,1},\ldots,d_{n,2^n-1})$ of length $2^n-1$ for any $n\geq 1$ by the following process: $D_1=(1)$ and for any $n\geq 2$,
\[D_n=(n,d_{n-1,1},n,d_{n-1,2},n,\ldots,n,d_{n-1,2^{n-1}-1},n).\]
For example,
\begin{align*}
D_2&=(2,1,2),\\
D_3&=(3,2,3,1,3,2,3),\\
D_4&=(4,3,4,2,4,3,4,1,4,3,4,2,4,3,4).
\end{align*}
In the sequence $D_n$, by construction, if $i<k$ and $d_{n,i}=d_{n,k}$, then there exists $j$ such that $i<j<k$ and $d_{n,j}<d_{n,i}$. So $D_n$ is a $n$-valid coloration of $J_{2^n-1}$.
By restriction to its first $k$ letters, it is a $n$-valid coloration of $J_k$, so $P_k(n)\neq 0$. \\

$\Longleftarrow$. Let us assume that $P_k(n)\neq 0$ and let us prove that $k\leq 2^n-1$. By hypothesis, there exists a valid $n$-coloration $f$ of $J_n$.
Let us prove by induction that $|f^{-1}(i)|\leq 2^{i-1}$ by induction on $i$. If $i=1$, by the $b$ condition, at most one element of $[k]$ is sent by $f$ to $1$, so $|f^{-1}(1)|\leq 1$. 
Let us assume the result at all ranks $<i$. By the $b$-condition, two elements sent by $f$ to $i$ are separated by at least one element of $f^{-1}([i-1])$. By the induction hypothesis,
\[|f^{-1}([i-1])|\leq 2^0+\cdots+2^{i-2}=2^{i-1}-1.\]
So, there can be at most $2^{i-1}$ elements in $f^{-1}(i)$. Finally, 
\[k=\sum_{i=1}^n |f^{-1}(i)|\leq \sum_{i=1}^n 2^{i-1}\leq 2^n-1. \qedhere\]
\end{proof}

\begin{remark}
In other term, the chromatic number of $J_k$, that is to say the minimal $n$ such that $J_k$ has a valid $n$-coloration, is the smallest integer greater than or equal to $\log_2(k+1)$. 
\end{remark}

\section{Other polynomial invariants}

\subsection{Homogeneous polynomial invariants}

The number of blocks of noncrossing partitions induces a graduation of $(\alg,m,\Delta)$ (notice that $\delta$ is not homogeneous for this graduation).
It is connected (that is to say $\alg_0=\K$), but its other homogeneous components are not finite dimensional.
With the help of \cite[Propositions 3.10 and 5.2]{Foissy40}, we can associate a homogeneous morphism $\phi_0: (\alg,m,\Delta)\longrightarrow (\K[X],m,\Delta)$ to the element $\mu \in  \alg_1^*$
such that $\mu(\pi)=1$ for any noncrossing partition $\pi$ with only one block (and $\mu$ takes the value $0$ on any other monomial in noncrossing partitions). 

\begin{defi}
Let $\pi$ be a noncrossing partition with $n$ blocks. We denote by $\ext(\pi)$ the number of linear extensions of $(\pi,\leq_\pi)$, that is to say bijections $f:\pi\longrightarrow [n]$ such that
\begin{align*}
&\forall b,b'\in \pi,&b\leq_\pi b'\Longrightarrow f(b)\leq f(b').
\end{align*}\end{defi}

By construction of the coproduct $\Delta$, for any noncrossing partition $\pi$ with $n$ blocks,
\[\mu^{\otimes n}\circ \Delta^{(n-1)}(\pi)=\ext(\pi).\]
Therefore, the morphism associated to $\mu$ is given by the following:

\begin{prop}
Let $\phi_0: \alg\longrightarrow \K[X]$ be the algebra morphism such that
\begin{align*}
&\forall \pi\in \ncp,&\phi_0(\pi)&=\frac{\ext(\pi)}{|\pi|!}X^{|\pi|}.
\end{align*}
Then $\phi_0$ is a bialgebra morphism from $(\alg,m,\Delta)$ to $(\K[X], m,\Delta)$. Considering the character $\lambda_0=\epsilon_\delta \circ \phi_0$,
\begin{align*}
&\forall \pi\in \ncp,&\lambda_0(\pi)&=\frac{\ext(\pi)}{|\pi|!}.
\end{align*}
Moreover, with the notations of Proposition \ref{prop1.2} and Theorem \ref{theo1.4}, $\phi_0=\phi_{\ncp}\leftsquigarrow\lambda_0$.
\end{prop}

See the Table at the end of the article for examples of $\lambda_0(\pi)$.

\subsection{Invertible characters}

In order to study invertible characters of $(\alg,m,\delta)$, let us introduce a second graduation. Let $\pi=\pi_1\cdot \ldots \cdot \pi_k$ be a monomial in noncrossing partitions. 
The integer $k$ is the length of $\pi$, and is denoted by $\lg(\pi)$. The number of blocks of $\pi$ is
\[|\pi|=|\pi_1|+\cdots+|\pi_k|.\]
Finally, the degree of $\pi$ is
\[\deg(\pi)=|\pi|-\lg(\pi).\]
Note that this belongs to $\N$. 

\begin{lemma}
With this degree, $(\alg,m,\delta)$ is a graded bialgebra. In particular, $\alg_0$ is the subalgebra generated by noncrossing partitions with only one block.
\end{lemma}

\begin{proof} Let $\pi,\pi'$ be monomials in noncrossing partitions. Then
\begin{align*}
|\pi\cdot \pi'|&=|\pi|+|\pi'|,&\lg(\pi\cdot \pi')&=\lg(\pi)+\lg(\pi').
\end{align*}
 we obtain that $\deg(\pi\cdot \pi')=\deg(\pi)+\deg(\pi')$: the product is homogeneous.  Let $\pi=\pi_1\cdots \pi_k$ be a monomial in noncrossing partitions, and let $\sim_i\in \eq_c[\pi_i]$ for any $i$. 
Then, by definition of $\eq_c[\pi_i]$,
 \begin{align*}
 |\pi_1\mid_\cdot \sim_1|+\cdots+|\pi_k\mid_\cdot \sim_k|&=|\pi_1|+\cdots+|\pi_k|=|\pi|,\\
 \lg(\pi_1\mid_\cdot \sim_1)+\cdots+\lg(\pi_k\mid_\cdot \sim_k)&=\cl(\sim_1)+\cdots+\cl(\sim_k),\\
|\pi_1/\sim_1|+\cdots+|\pi_k/\sim_k|&=\cl(\sim_1)+\cdots+\cl(\sim_k),\\
\lg(\pi_1/\sim_1)+\cdots+\lg(\pi_k/\sim_k)&=k,
 \end{align*}
 where $\cl(\sim_i)$ is the number of classes of $\sim_i$. Combining these equalities, we obtain that
 \[\deg(\pi_1/\sim_1\cdots \pi_k/\sim_k)+\deg(\pi_1\mid_\cdot \sim_1\cdots \pi_k\mid_\cdot \sim_k)=|\pi|-k=\deg(\pi).\]
 Therefore, the coproduct $\delta$ is homogeneous. 
\end{proof}

\begin{prop}\label{prop3.16}
Let $\lambda$ be a character of $(\alg,m)$. It is invertible for the convolution  product $\star $ associated to $\delta$ if, and only if, for any noncrossing partition $\pi$ with only one block, $\lambda(\pi)\neq 0$. 
\end{prop}

\begin{proof}
We apply \cite[Proposition 3.9]{Foissy45} with the preceding graduation, with the family of group-like elements of noncrossing partitions with only one block.
\end{proof}

In particular, for any noncrossing partition $\pi$ with only one block, $\lambda_0(\pi)=1$, so $\lambda_0$ is invertible. Its inverse is denoted by $\lambda_{\ncp}$. Therefore,
\begin{align}
\label{Eq13}\phi_{\ncp}&=\phi_0\leftsquigarrow \lambda_{\ncp}.
\end{align}
Consequently:

\begin{cor} \label{cor3.17}
For any noncrossing partition $\pi$,
\[\phi_\ncp(\pi)=\sum_{\sim\in \eq_c[\pi]} \lambda_{\ncp}(\pi\mid_\cdot \sim) \frac{\ext(\pi/\sim)}{\cl(\sim)!}X^{\cl(\sim)}.\]
In particular, the degree of $\phi_\ncp(\pi)$ is $|\pi|$, $\lambda_{\ncp}(\pi)$ is the coefficient of $X$ in $\phi_\ncp(\pi)$ and $\lambda_0(\pi)$ is the coefficient of $X^{|\pi|}$ in $\phi_\ncp(\pi)$.
\end{cor}

\begin{proof}
The formula for $\phi_\ncp$ is a direct consequence of (\ref{Eq13}). If $\sim\in \eq_c[\pi]$, then $\cl(\sim)\leq|\pi|$,
so the degree of $\phi_\ncp(\pi)$ is smaller than $|\pi|$. 
There is only one element $\sim_1\in \eq_c[\pi]$ such that $\cl(\sim_0)=|\pi|$: it is the equality of $\pi$.
Moreover, $\pi\mid\cdot \sim_1$ is a product of noncrossing partitions with a unique block, so $\lambda_{\ncp}(\pi \mid_\cdot \sim_1)=1$, whereas $\pi/\sim_1=\pi$. 
Therefore, the coefficient of $X^{|\pi|}$ is $\phi_\ncp(\pi)$ is $\lambda_0(\pi)$, which is non zero: the degree of $\phi_\ncp(\pi)$ is $|\pi|$.\\

There is only one element $\sim_0\in \eq_c[\pi]$ such that $\cl(\sim_0)=1$: for any $b,b'\in \pi$, $b\sim_0 b'$.
Moreover, $\pi\mid_\cdot \sim_0=\pi$ and $\pi/\sim_0$ has only one block, so $\dfrac{\ext(\pi/\sim_0)}{\cl(\sim_0)!}=1$: consequently, the coefficient of $X$ in $\phi_\ncp(\pi)$ is $\lambda_\ncp(\pi)$. 
\end{proof}

Let us now give an interpretation of the coefficient of $X^{|\pi|-1}$ in $\phi_{\ncp}(\pi)$.

\begin{defi}
Let $\pi$ a noncrossing partition and let $b$, $b'$ be two different blocks of $\pi$.
\begin{itemize}
\item We shall say that $(b,b')$ is a close pair if  $\max(b)<\min(b')$ and $b\sim_\pi b'$ (see Definition \ref{defi3.4}).
\item We shall say that $(b,b')$ is a nested pair if $b\leq_\pi b'$ and if for any $b''\in \pi$,
\[b\leq_\pi b''\leq_\pi b'\Longrightarrow b''=b \mbox{ or }b''=b'.\] 
\end{itemize}
If $(b,b')$ is a close or a nested pair of $\pi$, we denote by $\pi/(b,b')$ the partition which blocks are $b\sqcup b'$ and the other blocks of $\pi$. By definition of close and nested pairs, it is a noncrossing partition.
\end{defi}

\begin{cor}\label{cor3.19}
Let $\pi$ be a noncrossing partition with $n$ blocks. Then $\phi_\ncp(\pi)$ is a polynomial of degree $n$. The coefficient of $X^{n-1}$ is
\[-\frac{1}{(|\pi|-1)!}\left(\sum_{\mbox{\scriptsize $(b,b')$ close pair of $\pi$}}\ext(\pi/(b,b'))+\dfrac{1}{2}\sum_{\mbox{\scriptsize $(b,b')$ nested pair of $\pi$}}\ext(\pi/(b,b'))\right).\]
\end{cor}

\begin{proof}
Let $\sim\in \eq_c[\pi]$ such that $\cl(\sim)=|\pi|-1$. There exists exactly one class of $\sim$ made of two blocks $b$ and $b'$, and the other classes are made of only one blocks. Two possibilities can occur:
\begin{itemize}
\item Up to a permutation of $b$ and $b'$, $(b,b')$ is a close  pair. Then $\pi/\sim=\pi/(b,b')$.
Moreover, $\pi\mid_\cdot \sim$ is a monomial made of noncrossing partitions with only one block and a noncrossing partition with two non nested blocks. Hence, $\lambda_0(\pi\mid_\cdot \sim)=1$. 
As $\pi\mid_\cdot \sim$ is a skew-primitive element for $\delta$, $\lambda_\ncp(\pi\mid_\cdot \sim)=-1$.
\item Up to a permutation of $b$ and $b'$, $(b,b')$ is a nested  pair. Then $\pi/\sim=\pi/(b,b')$. 
Moreover, $\pi\mid_\cdot \sim$ is a monomial made of noncrossing partitions with only one block and a noncrossing partition with two nested blocks. 
Hence, $\lambda_0(\pi\mid_\cdot \sim)=\dfrac{1}{2}$. As $\pi\mid_\cdot \sim$ is a skew-primitive element for $\delta$, $\lambda_\ncp(\pi\mid_\cdot \sim)=-\dfrac{1}{2}$.
\end{itemize}
The result on the coefficient of $X^{|\pi|-1}$ immediately follows.
\end{proof}

\begin{example}\label{ex3.5}
\begin{enumerate}
\item Let us consider $J_n$ (see Notation \ref{not2.4}), with $n\geq 2$. Then $P_n=\phi_\ncp(J_n)$ is a monic polynomial of degree $n$, as $\lambda_0(J_n)=1$. 
Moreover, $J_n$ has no nested pair, any any $({i},{j})$ with $1\leq i<j\leq n$ is a close pair of $J_n$. Therefore, the coefficient of $X^{n-1}$ in $J_n$ is
\begin{align*}
b_{n-1,n}&=-\sum_{1\leq i<j\leq n} \frac{1}{j-i}=\sum_{1\leq i\leq n-1} \sum_{k=1}^{n-i} \frac{1}{k}=\sum_{1\leq j\leq n-1} \sum_{k=1}^j \frac{1}{k}=n\zeta_n(1)-n.
\end{align*}
Consequently, $((n-1)!b_{n-1,n})_{n\geq 0}$ is the sequence of generalized Stirling numbers, see Entry A001705 of the OEIS \cite{Sloane}. 
\item See the Table at the end of the article for examples of $\lambda_\ncp(\pi)$.
\end{enumerate}\end{example}

Let us give more results about $\lambda_\ncp(J_n)$.

\begin{prop}\label{prop3.20}
The sequence $((n-1)! \lambda_\ncp(J_n))_{n\geq 1}$ is (up to the signs) the infinitesimal generator of  $X(1+X)$, see Entries A005119  and A179199 of the OEIS \cite{Sloane}. Consequently, 
\begin{align*}
&\forall n\geq 2,&\lambda_\ncp(J_n)&=-\frac{1}{n-1}\sum_{i=1}^{\left[\frac{n}{2}\right]} \binom{n-i+1}{i+1}\lambda_\ncp(J_{n-i}).
\end{align*}\end{prop}

\begin{proof}
This is a direct consequence of Proposition \ref{prop3.11}, Remark \ref{rk3.3} and the recursive formula given in Entry A005119 of the OEIS \cite{Sloane}.
\end{proof}

\begin{example}
\[\begin{array}{|c||c|c|c|c|c|c|c|c|c|c|}
\hline n&1&2&3&4&5&6&7&8&9&10\\
\hline \hline &&&&&&&&&&\\[-2mm]
\lambda_\ncp(J_n)&1&-1&\dfrac{3}{2}&-\dfrac{8}{3}&\dfrac{31}{6}&-\dfrac{157}{15}&\dfrac{649}{30}&-\dfrac{9427}{210}&\dfrac{19423}{210}&-\dfrac{6576}{35}\\[-2mm]
&&&&&&&&&&\\
\hline\end{array}\]\end{example}

\begin{remark}
It could be conjectured from these values that the sign of $\lambda_\ncp(J_n)$ is $(-1)^{n+1}$ for any $n$. This is false. The first counterexample is 
\[\lambda_\ncp(J_{29})=-\dfrac{37449570182565026}{37182145}.\]
Other counterexamples can be found for $n=30$, $33$, $34$, $38$, $39$, $42$, $43$, $47$, $48$, $51$, $52$, $55$, $56$. 
\end{remark}

\subsection{Linear extensions}

Let us first generalize some combinatorial notions from noncrossing partitions to monomial of noncrossing partitions.

\begin{defi}
Let $P=\pi_1\cdot \ldots \cdot \pi_k$ be a monomial in noncrossing partitions. The set of blocks of $P$ is
\[\bloc(P)=\pi_1\sqcup \ldots \sqcup \pi_k.\]
We also put $|P|=|\bloc(P)|$.
The set $\bloc(P)$ is partially ordered by $\leq_P=\leq_{\pi_1}\sqcup \ldots \sqcup \leq_{\pi_k}$. An ideal of $P$ is a subset $I$ of $\bloc(P)$ such that
\begin{align*}
&\forall b,b'\in \bloc(P),&(b\in I\mbox{ and }b\leq_P b')\Longrightarrow (b'\in I).
\end{align*}\end{defi}

With this definition, we can write, for any monomial in noncrossing partitions,
\[\Delta(P)=\sum_{\mbox{\scriptsize $J$ ideal of $P$}} P_{\mid \bloc(P)\setminus J}\otimes P_{\mid_\cdot J}.\]

\begin{defi}
Let $P$ be a monomial in noncrossing partitions and let $n\geq 1$.
\begin{enumerate}
\item A $n$-linear extension of $P$ is a map $f:\bloc(P)\longrightarrow [n]$ such that
\begin{align*}
&\forall b,b'\in \bloc(P),&b\leq_P b'\Longrightarrow f(b)\leq f(b').
\end{align*}
The set of $n$-linear extensions of $P$ is denoted by $\lin(P,n)$.
\item A strict $n$-linear extension of $P$ is a map $f:\bloc(P)\longrightarrow [n]$ such that
\begin{align*}
&\forall b,b'\in \bloc(P),&(b\leq_P b'\mbox{ and }b\neq b')\Longrightarrow f(b)<f(b').
\end{align*}
The set of strict $n$-linear extensions of $P$ is denoted by $\linstr(P,n)$.
\end{enumerate}\end{defi}

\begin{remark}
Obviously, $\linstr(P,n)\subseteq \lin(P,n)$ for any $n\geq 1$.
\end{remark}

\begin{theo}\label{theo3.23}
\begin{enumerate}
\item For any monomial $P$ in noncrossing partitions, there exists a unique $\Lambda(P)(X) \in \K[X]$, such that
\begin{align*}
&\forall n\geq 1,&\Lambda(P)(n)&=|\lin(P,n)|.
\end{align*}
Moreover, the extension of $\Lambda$ as a linear map from $\alg$ to $\K[X]$ is a Hopf algebra map from $(\alg,m,\Delta)$ to $(\K[X],m,\Delta)$.
\item For any monomial $P$ in noncrossing partitions, there exists a unique $\Lambda^s(P)(X) \in \K[X]$, such that
\begin{align*}
&\forall n\geq 1,&\Lambda^s(P)(n)&=|\linstr(P,n)|.
\end{align*}
Moreover, the extension of $\Lambda^s$ as a linear map from $\alg$ to $\K[X]$ is a Hopf algebra map from $(\alg,m,\Delta)$ to $(\K[X],m,\Delta)$.
\end{enumerate}\end{theo}

\begin{proof}
We only prove the first point. The proof of the second point is very similar. Let us first consider a monomial $P$ in noncrossing partitions. 
For any $k\geq 1$, let us denote by $l_k(P)$ the number of surjective $k$-linear extensions of $P$. Note that if $k>|P|$, $l_k(P)=0$. Then, for any $n\geq 1$,
\[|\lin(P,n)|=\sum_{k=1}^{|P|} l_k(P) \binom{n}{k}.\]
We therefore take
\[\Lambda(P)=\sum_{k=1}^{|P|} l_k(P)H_k(X)\in \K[X].\]
Then, for any $n\geq 1$, $\lambda(P)(n)=|\lin(P,n)|$.\\

By convention, $\Lambda(1)=1$. If $P\neq 1$, then 
\[\Lambda(P)(0)=\sum_{k=1}^{|P|} l_k(P) \binom{0}{k}=0,\]
so $\varepsilon_\Delta \circ \Lambda(P)=0=\varepsilon_\Delta(P)$. Let us now prove the compatibility of $\Lambda$ with the products. Let  $P,Q$ be two monomials in noncrossing partitions. 
Then $\bloc(P\cdot Q)=\bloc(P)\sqcup \bloc(Q)$ and $\leq_{P\cdot Q}=\leq_P\sqcup \leq_Q$. As a consequence, for any $n\geq 1$,
\[\lin(P\cdot Q,n)=\{f\sqcup g\mid (f,g)\in \lin(P,n)\times \lin(Q,n)\}.\]
Taking the cardinalities, we obtain that for any $n\geq 1$,
\[\Lambda(P\cdot Q)(n)=\Lambda(P)(n)\Lambda(Q)(n).\]
This gives $\Lambda(P\cdot Q)=\Lambda(P)\Lambda(Q)$.\\

Let us finally prove the compatibility of $\Lambda$ with the coproducts $\Delta$. We fix a monomial $P$ in noncrossing partitions. Let $m,n\geq 1$.  If $f\in \lin(P,m+n)$, then:
\begin{itemize}
\item $f^{-1}(\{m+1,\ldots,m+n\})$ is an ideal of $P$, which we denote $I_f$.
\item $f_{\mid \bloc(P)\setminus I_f}$ is a $m$-linear extension of $P_{\mid \bloc(P)\setminus I_f}$.
\item $f_{\mid I}-m$ is a $n$-linear extension of $P_{\mid _\cdot I_f}$.
\end{itemize}
Conversely, if $I$ is an ideal of $P$, $f\in \lin(P_{\mid \bloc(P)\setminus I},m)$ and $g\in \lin(P_{\mid_\cdot I},n)$ the the following belongs to $\lin(P,m+n)$:
\[\left\{\begin{array}{rcl}
\bloc(P)&\longrightarrow&[m+n]\\
b&\longmapsto&\begin{cases}
f(b)\mbox{ if }b\notin I,\\
g(b)+m\mbox{ if }b\in I.
\end{cases}
\end{array}\right.\] 
Taking the cardinalities, we obtain that
\[\Lambda(P)(m+n)=\sum_{\mbox{\scriptsize $I$ ideal of $P$}}\Lambda(P_{\mid \bloc(P)\setminus I})(m)\Lambda(P_{\mid_\cdot I})(n).\]
Identifying $\K[X,Y]$ and $\K[X]^{\otimes 2}$, we obtain that for any $m,n\geq 1$,
\[\Delta\circ \Lambda(P)(m,n)=(\Lambda \otimes \Lambda)\circ \Delta(P)(m,n).\]
Therefore, $\Delta\circ \Lambda=(\Lambda \otimes \Lambda)\circ \Delta$.
\end{proof}

\begin{cor}\label{cor3.24}
Let $\lambda$ and $\lambda^s$ be the characters of $\alg$ defined by
\begin{align*}
\lambda&:\left\{\begin{array}{rcl}
\alg&\longrightarrow&\K\\
\pi\in \ncp&\longmapsto&1,
\end{array}\right.&
\lambda^s&:\left\{\begin{array}{rcl}
\alg&\longrightarrow&\K\\
\pi\in \ncp&\longmapsto&\begin{cases}
1\mbox{ if }\pi=\base(\pi),\\
0\mbox{ otherwise}.
\end{cases}
\end{array}\right.
\end{align*}
Then, with the notations of Proposition \ref{prop1.2} and Theorem \ref{theo1.4},
\begin{align}
\label{Eq14}
\Lambda&=\phi_\ncp\leftsquigarrow \lambda, &\Lambda^s&=\phi_\ncp\leftsquigarrow \lambda^s.
\end{align}\end{cor}

\begin{proof}
We use the bijection of Theorem \ref{theo1.4}. We define two elements of $\Char(\alg)$ by
\begin{align*}
&\forall x\in \alg,&\lambda(x)&=\Lambda(x)(1),&\lambda^s(x)&=\Lambda^s(x)(1).
\end{align*}
Theorem \ref{theo1.4} gives (\ref{Eq14}).
It remains to prove the formulas for $\lambda$ and $\lambda^s$. Let $\pi$ be a noncrossing partition. As the unique map $f:\bloc(\pi)\longrightarrow [1]$ is indeed a linear extension of $\pi$, 
\[\lambda(\pi)=\Lambda(\pi)(1)=|\lin(\pi,1)|=1.\]
Moreover, the unique map $f:\bloc(\pi)\longrightarrow [1]$ is a strict linear extension of $\pi$ if, and only if, $\pi=\base(\pi)$, so
\[\lambda^s(\pi)=\Lambda^s(\pi)(1)=|\linstr(\pi,1)|=\begin{cases}
1\mbox{ if }\pi=\base(\pi),\\
0\mbox{ otherwise}. 
\end{cases}\qedhere\] \end{proof}

\begin{prop}[\textbf{Duality principle}] For any $\pi\in \ncp$,
\[\Lambda(\pi)(X)=(-1)^{|\pi|}\Lambda^s(\pi)(-X).\]
\end{prop}

\begin{proof}
Let us consider the two maps
\begin{align*}
\phi_1&:\left\{\begin{array}{rcl}
\K[X]&\longrightarrow&\K[X]\\
P(X)&\longmapsto&P(-X),
\end{array}\right.&
\phi_2&:\left\{\begin{array}{rcl}
\alg&\longrightarrow&\alg\\
\pi\in \ncp&\longrightarrow&(-1)^{|\pi|}\pi.
\end{array}\right.\end{align*}
The map $\phi$ is a bialgebra isomorphism (it is in fact the antipode of $\K[X]$) and, as the bialgebra $(\alg,m,\Delta)$ is graded by the number of blocks, $\phi_2$ is also a bialgebra map.
By composition, $\phi=\phi_1\circ \Lambda^s\circ \phi_2$ is a bialgebra map. Moreover, for any $\pi\in \ncp$,
\[\phi(\pi)=(-1)^{|\pi|}\Lambda^s(\pi)(-X).\]
By Theorem \ref{theo1.4}, there exists a (unique) $\nu\in \Char(\alg)$ such that $\phi=\phi_\ncp \leftsquigarrow\nu$. Is is now enough to prove that $\nu=\lambda$. For any $\pi\in \ncp$,
\[\nu(\pi)=\phi(\pi)(1)=(-1)^{|\pi|}\Lambda^s(\pi)(-1),\]
so it is enough to prove that for any noncrossing partition $\pi$, $\Lambda^s(\pi)(-1)=(-1)^{|\pi|}$.We proceed by induction on $|\pi|$. If $|\pi|=1$, then $\Lambda^s(\pi)=X$ and the result holds. 
Let us assume that $|\pi|\geq 2$ and that the result holds for any noncrossing partition $\pi'$ with $|\pi'|<|\pi|$.
If $f\in \lin^s(\pi)$, observe that $f^{-1}(1)\subseteq \base(\pi)$. Conversely, if $B\subseteq \base(\pi)$ and $g\in \lin^s(\pi_{\mid_\cdot \pi\setminus B},n)$, one defines $f\in \lin^s(\pi,n+1)$ by
\begin{align*}
&\forall b\in \pi,&f(b)=\begin{cases}
1\mbox{ if }b\in B,\\
g(b)+1\mbox{ otherwise}.
\end{cases}\end{align*}
Consequently,
\[\Lambda^s(\pi)(X+1)=\sum_{B\subseteq \base(\pi)} \Lambda^s(\pi_{\mid_\cdot \pi\setminus B})(X).\]
Taking $X=-1$ and using the induction hypothesis on $\pi_{\mid_\cdot \pi\setminus B}$ when $B\neq \emptyset$, we obtain
\begin{align*}
\Lambda^s(\pi)(0)&=\epsilon_\delta(\pi)=0\\
&=\Lambda^s(\pi)(-1)+\sum_{\emptyset\subsetneq B\subseteq \base(\pi)} \Lambda^s(\pi_{\mid_\cdot \pi\setminus B})(-1)\\
&=\Lambda^s(\pi)(-1)+\sum_{\emptyset\subsetneq B\subseteq \base(\pi)} (-1)^{|\pi|-|B|}\\
&=\Lambda^s(\pi)(-1)+(-1)^{|\pi|}\left(\sum_{B\subseteq \base(\pi)}(-1)^{|B|} -1\right)\\
&=\Lambda^s(\pi)(-1)-(-1)^{|\pi|}.
\end{align*}
Note that $\displaystyle \sum_{B\subseteq \base(\pi)}(-1)^{|B|}=0$, as $\base(\pi)\neq \emptyset$. Finally, the result holds for $\pi$. 
\end{proof}

\begin{cor}\label{cor3.26}
For any $\pi\in \ncp$, $\lambda(\pi)=(-1)^{|\pi|} \lambda^s\circ S(\pi)$.
\end{cor}

\begin{proof}
From the proof of the duality principle, with the same notations, $\Lambda=\phi_1\circ \Lambda^s\circ \phi_2=S\circ \Lambda^s \circ \phi_2$. 
As $\Lambda^s$ is a bialgebra morphism, $\Lambda=\Lambda^s\circ S\circ \phi_2$. Evaluating this in $\pi$ and then taking $X=1$ gives the result. 
\end{proof}

\begin{notation}
By Proposition \ref{prop3.16}, both $\lambda$ and $\lambda^s$ are invertible for the convolution product $\star$, dual to the coproduct $\delta$. We denote by $\mu$ and $\mu^s$ their respective inverses. 
We immediately obtain
\begin{align*}
\phi_\ncp&=\Lambda \leftsquigarrow \mu,&\phi_\ncp&=\Lambda^s \leftsquigarrow \mu^s.
\end{align*}\end{notation}

Let us first compute $\mu^s$.

\begin{prop}\label{prop3.27}
For any $\pi\in \ncp$, $\mu^s(\pi)=\begin{cases}
0\mbox{ if }\pi\neq \base(\pi),\\
(-1)^{|\pi|-1} \mbox{ if }\pi=\base(\pi).
\end{cases}$
\end{prop}

\begin{proof}
For any noncrossing partition $\pi$, 
\[\lambda^s\star \mu^s(\pi)=\epsilon_\delta(\pi)=\sum_{\sim\in \eq_c[\pi]} \lambda^s(\pi/\sim)\mu^s(\pi).\]
We denote by $\pi_1,\ldots,\pi_n$ the connected components of $\pi$, totally ordered from left to right. 
As $\lambda^s$ only charges noncrossing partitions equal to their bases, in this sum we restrict ourselves with $\sim$ which classes are  unions of consecutive $\pi_j$, which gives
\[\epsilon_\delta(\pi)=\sum_{k=1}^n \sum_{\substack{n_1+\cdots+n_k=n,\\ n_1,\ldots,n_k\geq 1}}\mu^s(\pi_1\ldots \pi_{n_1})\ldots \mu^s(\pi_{n_1+\cdots+n_{k-1}+1}\ldots\pi_{n_1+\cdots+n_k}).\]

Let us assume firstly that $\pi\neq \base(\pi)$ and let us show that $\mu^s(\pi)=0$. Observe that $\epsilon_\delta(\pi)=0$. We proceed by induction on $n$. If $n=1$, we immediately obtain that $0=\mu^s(\pi)$.
If $n\geq 2$, for any $n_1,\ldots,n_k\geq 1$ such that $n_1+\cdots+n_k=n$, at least one of the noncrossing partitions $P_{n_1+\cdots+n_i+1}\ldots P_{n_1+\cdots+n_{i+1}}$ is not equal to its base. 
Therefore, if $k\geq 2$, by the induction hypothesis, it vanishes under $\mu^s$. We obtain $0=\mu^s(\pi)+0$, which ends this proof by induction.\\

Let us now assume that $\pi=\base(\pi)$. We define a sequence $(a_n)_{n\geq 1}$ by
\begin{align*}
a_n&=\begin{cases}
1\mbox{ if }n=1,\\
\displaystyle -\sum_{k=2}^n \sum_{\substack{n_1+\cdots+n_k=n,\\ n_1,\ldots,n_k\geq 1}}a_{n_1}\ldots a_{n_k} \mbox{ if }n\geq 2.
\end{cases}\end{align*}
Let us prove that $\mu^s(\pi)=a_{|\pi|}$ by induction on $|\pi|$. If $|\pi|=1$, then $\mu^s(\pi)=1=a_1$. If $|\pi|\geq 2$, the induction hypothesis gives
\begin{align*}
0&=\mu^s(\pi)+\sum_{k=2}^n \sum_{\substack{n_1+\cdots+n_k=n,\\ n_1,\ldots,n_k\geq 1}}\mu^s(\pi_1\ldots \pi_{n_1})\ldots \mu^s(\pi_{n_1+\cdots+n_{k-1}+1}\ldots\pi_{n_1+\cdots+n_k})\\
&=\mu^s(\pi)+\sum_{k=2}^n \sum_{\substack{n_1+\cdots+n_k=n,\\ n_1,\ldots,n_k\geq 1}}a_{n_1}\ldots a_{n_k}\\
&=\mu^s(\pi)-a_n,
\end{align*}
so $\mu^s(\pi)=a_n$. It remains to prove that $a_n=(-1)^{n-1}$ for any $n\geq 1$. Let us consider the formal series $A=\displaystyle \sum_{k=1}^\infty a_kX^k \in \mathbb{X}[[X]]$.
 By definition of the sequence $(a_n)_{n\geq 1}$, 
\begin{align*}
A&=X-\sum_{n=2}^\infty \sum_{k=2}^\infty  \sum_{\substack{n_1+\cdots+n_k=n,\\ n_1,\ldots,n_k\geq 1}}a_{n_1}\ldots a_{n_k}X^n\\
&=X-\sum_{k=2}^\infty  \sum_{\substack{n_1+\cdots+n_k=n,\\ n_1,\ldots,n_k\geq 1}}a_{n_1}X^{n_1}\ldots a_{n_k}X^{n_k}\\
&=X-\sum_{k=2}^\infty A^k\\
&=X-\frac{A^2}{1-A}.
\end{align*}
Hence, $\dfrac{A}{1-A}=X$ and $A=\dfrac{X}{1+X}$, which finally implies that $a_n=(-1)^{n-1}$ for any $n\geq 1$. 
\end{proof}

This finally gives:

\begin{prop}
For any noncrossing partition $\pi$, $\displaystyle \phi_\ncp(\pi)=\sum_{\sim \in \eq_c[\pi]} (-1)^{\bl(\pi\mid_\cdot \sim)-1}\Lambda^s(\pi/\sim)$.
\end{prop}

In order to describe $\mu$, we shall introduce a family of characters:

\begin{prop}\label{prop3.29}
For any $q\in \K$, we define $\gamma_q\in \Char(\alg)$ by
\begin{align*}
\gamma_q&:\left\{\begin{array}{rcl}
\alg&\longrightarrow&\K\\
\pi\in \ncp&\longmapsto&\begin{cases}
q\mbox{ if }|\pi|=1,\\
0\mbox{ otherwise}.
\end{cases}\end{array}\right.\end{align*}
In particular, $\gamma_1=\epsilon_\delta$. Then:
\begin{enumerate}
\item For any $\lambda\in \Char(\alg)$, for any $\pi\in \ncp$, $\lambda \star \gamma_q(\pi)=q^{|\pi|} \lambda(\pi)$.
\item For any $q,q'\in \K$, $\gamma_{q'}\star \gamma_q=\gamma_{qq'}$. 
\item If $q\neq 0$, $\gamma_q$ is invertible for $\star$, and $\gamma_q^{\star-1}=\gamma_{q^{-1}}$.
\item  For any $\lambda\in \Char(\alg)$, for any $\pi\in \ncp$, $\gamma_q\star \lambda (\pi)=q \lambda(\pi)$.
\end{enumerate}\end{prop}

\begin{proof}
1. For any $\pi\in \ncp$,
\[\lambda\star \gamma_q(\pi)=\sum_{\sim\in \eq_c[\pi]}\lambda(\pi/\sim) \gamma_q(\pi\mid\sim).\]
Let us consider $\sim\in \eq_c[\pi]$ such that $\gamma_q(\pi\mid\sim)\neq 0$: $\pi\mid\sim$ is a monomial of noncrossing partitions with a single block, which means that $\sim$ is the equality $\sim_1$ of $\pi$. Therefore,
\[\lambda\star \gamma_q(\pi)=\lambda(\pi/\sim_1)\gamma_q(\pi\mid \sim_1)=\lambda(\pi)q^{|\pi|}.\]

2. In the particular case where $\lambda=\gamma_{q'}$, for any $\pi\in \ncp$,
\[\gamma_{q'}\star\gamma_q(\pi)=q^{|\pi|} \gamma_{q'}(\pi)=\begin{cases}
qq'\mbox{ if }|\pi|=1,\\
0\mbox{ otherwise}.
\end{cases}\]
So $\gamma_{q'}*\gamma_q=\gamma_{qq'}$.\\

3. Easy consequence of 2. \\

4. For any $\pi\in \ncp$,
\[\gamma_q \star \lambda(\pi)=\sum_{\sim\in \eq_c[\pi]}\gamma_q(\pi/\sim) \lambda(\pi\mid\sim).\]
Let us consider $\sim\in \eq_c[\pi]$ such that $\gamma_q(\pi/\sim)\neq 0$: $\pi/\sim$ is a noncrossing partition with a single block, which means that $\sim$ is equivalence $\sim_0$ on $\pi$ with a single class. 
Therefore,
\[\gamma_q \star \lambda(\pi)=\gamma_q(\pi/\sim_0)\lambda(\pi\mid_\cdot \sim_0)=q\lambda(\pi). \qedhere\]
\end{proof}

\begin{lemma}\label{lem3.30}
For any $\lambda\in \Char(\alg)$, $\lambda\circ S=\mu_\ncp \star \lambda$.
\end{lemma}

\begin{proof}
Recall that $S=(\mu_\ncp \otimes \id_{\alg})\circ \delta$. Therefore,
\begin{align*}
\lambda \circ S&=\lambda \circ (\mu_\ncp \otimes \id_{\alg})\circ \delta=(\mu_\ncp \otimes \lambda)\circ \delta=\mu_\ncp\star \lambda. \qedhere
\end{align*}\end{proof}

\begin{lemma}\label{lem3.31}
$\mu_\ncp\star\mu_\ncp=\epsilon_\delta$.
\end{lemma}

\begin{proof}
By Lemma \ref{lem3.30}, for $\lambda=\epsilon_\delta$, 
\[\epsilon_\delta\circ S=\mu_\ncp\star \epsilon_\delta=\mu_\ncp.\]
Still by this Lemma, with $\lambda=\mu_\ncp$,
\[\mu_\ncp \star \mu_\ncp=\mu_\ncp \circ S=\epsilon_\delta\circ S^2=\epsilon_\delta.\]
Indeed, as $\alg$ is commutative, $S^2=\id_{\alg}$. 
\end{proof}

\begin{prop}\label{propr3.32}
In $\Char(\alg)$,
$\lambda=\mu_\ncp\star \lambda^s\star \gamma_{-1}$ and $\mu=\gamma_{-1}\star \mu^s \star\mu_\ncp$. Moreover,
\begin{align*}
&\forall \pi\in \ncp,&\mu(\pi)&=\sum_{\substack{\sim \in \eq_c[\pi],\\ \pi/\sim=\base(\pi/\sim)}}(-1)^{\mathrm{cl}(\sim)}\mu_\ncp(P\mid_\cdot \sim),
\end{align*}
where $\mathrm{cl}(\sim)$ is the number of classes of $\sim$. 
\end{prop}

\begin{proof}
By Corollary \ref{cor3.26}, Lemma \ref{lem3.30} and Proposition \ref{prop3.29}-1, $\lambda=\mu_\ncp \star \lambda^s\star \gamma_{-1}$. 
By proposition \ref{prop3.29}-3 and Lemma \ref{lem3.31}, the inverses (for $\star$) of $\gamma_{-1}$ and $\mu_\ncp$ are themselves, so
\[\mu=\lambda^{\star-1}=\gamma_{-1}^{\star-1}\star (\lambda^s)^{\star-1}\star \mu_\ncp^{\star-1}=\gamma_{-1}\star \mu^s \star \mu_\ncp. \]
By Proposition \ref{prop3.29}-4, for any noncrossing partition $\pi$,
\[\mu(\pi)=-\mu^s\star \mu_\ncp(\pi)=-\sum_{\sim \in \eq_c[\pi]}\mu^s(\pi/\sim)\mu_\ncp(P\mid_\cdot \sim).\]
The formula then comes from Proposition \ref{prop3.27}.
\end{proof}

\begin{example}
See  the Table at the end of the article for examples of $\mu^s(\pi)$ and $\mu(\pi)$.
\end{example}

\begin{cor}
Let $\pi\in \ncp$, such that $\pi=\base(\pi)$. Then $\mu(\pi)=(-1)^{|\pi|-1}\cat_{|\pi|-1}$. 
\end{cor}

\begin{proof}
Let $\pi\in \ncp$, such that $\pi=\base(\pi)$. The blocks of $\pi$ are denoted by $\pi_1,\ldots,\pi_n$, totally ordered from left to right.
If $\sim\in \eq_c[\pi]$ such that $\pi/\sim=\base(\pi/\sim)$, then the classes of $\sim$ are formed by consecutive $\pi_i$. This gives
\begin{align*}
\mu(\pi)&=\sum_{k=1}^n \sum_{\substack{n_1+\cdots+n_k=n,\\ n_1,\ldots,n_k\geq 1}} (-1)^k \mu_\ncp(\pi_1\ldots \pi_{n_1})\ldots \mu_\ncp(\pi_{n_1+\cdots+n_{k-1}+1}\ldots \pi_{n_1+\cdots+n_k})\\
&=\sum_{k=1}^n \sum_{\substack{n_1+\cdots+n_k=n,\\ n_1,\ldots,n_k\geq 1}} (-1)^{k+n_1+\cdots+n_k}\cat_{n_1}\ldots \cat_{n_k}.
\end{align*}
 We put, for any $n\geq 1$,
\[a_n=\sum_{k=1}^n \sum_{\substack{n_1+\cdots+n_k=n,\\ n_1,\ldots,n_k\geq 1}} (-1)^{k+n_1+\cdots+n_k}\cat_{n_1}\ldots \cat_{n_k},\]
so that $\mu(\pi)=a_{|\pi|}$. We consider the formal series
\begin{align*}
A&=\sum_{n=1}^\infty a_nX^n\in \mathbb{Q}[[X]],&C=\sum_{n=1}^\infty \cat_nX^n=\frac{1-\sqrt{1-4X}}{2X}-1 \in \mathbb{Q}[[X]].
\end{align*}
By definition of $a_n$,
\[A=\sum_{k=1}^\infty (-1)^k C(-X)^k=\frac{C(-X)}{1+C(-X)}.\]
Replacing the expression of $C$ in this formula and working out, we obtain
\[A=\frac{1-\sqrt{1+4X}}{2}=X(1+C(-X))=\sum_{n=1}^\infty (-1)^{n-1}\cat_{n-1}X^n,\]
which gives the result for $\mu(\pi)$. 
\end{proof}

\section{Combinatorial morphisms from noncrossing partitions to graphs}

We now turn to the bialgebras of hypergraphs or mixed graphs defined in \cite{Foissy44,Foissy45}. We obtain the following negative results:

\begin{prop}\label{prop4.1}
\begin{enumerate}
\item Let $\leftthreetimes \in \{\cap,\subset\}$.  There exists no double bialgebra morphism $\psi$ from $\alg$  to the double bialgebra of hypergraphs 
$(\calF[\bfH],m,\Delta,\delta^{(\leftthreetimes)})$ of \cite{Foissy44},  such that $\psi(J_3)$ is a sum of  hypergraphs.
\item There exists no double bialgebra morphism $\psi$ from $\alg$ to the double bialgebra of mixed graphs $\calH_\bfG$ of \cite{Foissy45}, such that $\psi(J_3)$ is a mixed graph.
\end{enumerate}\end{prop}

\begin{proof}
1. Let $\psi$ be such a morphism. Then $\phi_\ncp\circ \psi:\alg\longrightarrow \K[X]$ is a double bialgebra morphism by composition, so is equal to $\phi_{chr}$. 
Hence, there exists a sum of hypergraphs $G_1+\cdots+G_k=\psi(J_3)$ such that
\[\phi_{chr}(G_1+\cdots+G_k)=X^3 - \frac{5X^2}{2} + \frac{3X}{2}.\]
This contradicts \cite[Proposition 2.5]{Foissy44}, stating that the coefficients of $\phi_{chr}(H)$ are integers for any hypergraph. \\

2. Let $\psi$ be such a morphism. Then $\phi_\ncp\circ \psi:\alg\longrightarrow \K[X]$ is a double bialgebra morphism by composition, so is equal to $\phi_{chr}$. 
Hence, there exists a mixed graph $G=\psi(J_3)$ such that
\[\phi_{chr}(G)=X^3 - \frac{5X^2}{2} + \frac{3X}{2}.\]
By  \cite[Proposition 3.14]{Foissy45}, $G$ has no cycle and has three vertices. It is not a graph, as for any graph $G$, the coefficients of $\phi_\ncp(G)=P_{chr}(G)$ are integers. 
As it is not divisible by $X^2$, $G$ is connected. It is not one of the mixed graphs of \cite[Example 3.2]{Foissy45}. So $G$ is one of the four following mixed graphs:
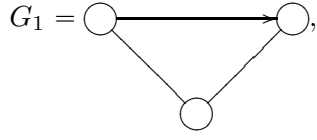
\begin{align*}
G_1&=\xymatrix{\rond{}\ar@{-}[rd]\ar[rr]&&\rond{}\ar@{-}[ld]\\
&\rond{}&},&
G_2&=
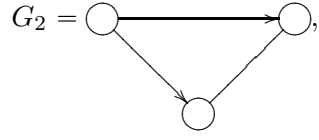
\xymatrix{\rond{}\ar[rd]\ar[rr]&&\rond{}\ar@{-}[ld]\\
&\rond{}&},\\
G_3&=
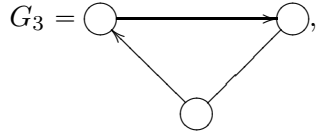
\xymatrix{\rond{}\ar[rr]&&\rond{}\ar@{-}[ld]\\
&\rond{}\ar[ul]&},&
G_4&=
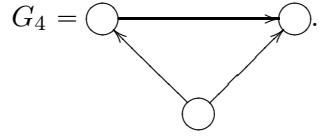
\xymatrix{\rond{}\ar[rr]&&\rond{}\\
&\rond{}\ar[ul]\ar[ru]&}.
\end{align*}
But for any of these graphs, $\phi_{chr}(G_i)(2)=0$, whereas $\phi_{chr}(G)(2)=1$. So such a $\psi$ does not exist. 
\end{proof}

\begin{remark}
Of course, there exist double bialgebra morphisms from $\alg$ to $\calF[\bfH]$ or to $\calH_\bfG$. For example, observe that the two following define double bialgebra morphisms:
\begin{align*}
&\left\{\begin{array}{rcl}
\K[X]&\longrightarrow&\calF[\bfH]\\
X&\longmapsto&\bullet,
\end{array}\right.&
&\left\{\begin{array}{rcl}
\K[X]&\longrightarrow&\calH_\bfG\\
X&\longmapsto&\bullet.
\end{array}\right.
\end{align*}
The composition with $\phi_\ncp$ gives two double bialgebra morphisms, which of course do not satisfy the combinatorial conditions of Proposition \ref{prop4.1}.
\end{remark}

\newpage

\begin{table}[h]\[\begin{array}{|c||c|c|c|c|c|c|}
\hline \pi&\phi_\ncp(\pi)&\mu_\ncp(\pi)&\lambda_0(\pi)&\lambda_\ncp(\pi)&\mu^s(\pi)&\mu(\pi)\\
\hline\hline \ncun&X&-1&1&1&1&1\\
\hline \ncdeuxun&X&-1&1&1&1&1\\
\hline \ncdeuxdeux&X(X-1)&2&1&-1&-1&-1\\
\hline \nctroisun&X&-1&1&1&1&1\\
\hline \nctroisdeux&X(X-1)&2&1&-1&-1&-1\\
\hline \nctroistrois&X(X-1)&2&1&-1&-1&-1\\
\hline &&&&&\\[-4mm]
\nctroisquatre&\dfrac{X(X-1)}{2}&1&\dfrac{1}{2}&\dfrac{1}{2}&0&-1\\[2mm]
\hline &&&&&\\[-4mm]
\nctroiscinq&X(X-1)\left(X-\dfrac{3}{2}\right)&-5&1&\dfrac{3}{2}&1&2\\
\hline \ncquatreun&X&-1&1&1&1&1\\
\hline \ncquatredeux&X(X-1)&2&1&-1&-1&-1\\
\hline \ncquatretrois&X(X-1)&2&1&-1&-1&-1\\
\hline &&&&&\\[-4mm]
\ncquatrequatre&\dfrac{X(X-1)}{2}&1&\dfrac{1}{2}&-\dfrac{1}{2}&0&-1\\[2mm]
\hline &&&&&\\[-4mm]
\ncquatrecinq&\dfrac{X(X-1)}{2}&1&\dfrac{1}{2}&-\dfrac{1}{2}&0&-1\\[2mm]
\hline \ncquatresix&X(X-1)&2&1&-1&-1&-1\\
\hline &&&&&\\[-4mm]
\ncquatresept&\dfrac{X(X-1)}{2}&1&\dfrac{1}{2}&-\dfrac{1}{2}&0&-1\\[2mm]
\hline &&&&&\\[-4mm]
\ncquatrehuit &X(X-1)\left(X-\dfrac{3}{2}\right)&-5&1&\dfrac{3}{2}&1&2\\[2mm]
\hline &&&&&\\[-4mm]
\ncquatreneuf&X(X-1)\left(X-\dfrac{3}{2}\right)&-5&1&\dfrac{3}{2}&1&2\\[2mm]
\hline &&&&&\\[-4mm]
\ncquatredix&X(X-1)\left(X-\dfrac{3}{2}\right)&-5&1&\dfrac{3}{2}&1&2\\[2mm]
\hline \ncquatreonze&X(X-1)^2&-2&1&1&0&3\\
\hline \ncquatredouze&X(X-1)^2&-2&1&1&0&3\\
\hline &&&&&\\[-4mm]
\ncquatretreize&\dfrac{X(X-1)(X-2)}{3}&-2&\dfrac{1}{3}&\dfrac{2}{3}&0&2\\[2mm]
\hline &&&&&\\[-4mm]
\ncquatrequatorze&X(X-1)(X-2)\left(X-\dfrac{4}{3}\right)&14&1&-\dfrac{8}{3}&-1&-5\\[2mm]
\hline\end{array}\]
\caption{examples}
\end{table}

\newpage

\bibliographystyle{amsplain}
\bibliography{biblio}

\end{document}